\documentclass[11pt]{amsart}
\usepackage[margin=3cm]{geometry} 
\usepackage{amsmath}
\usepackage{subcaption}
\usepackage{amssymb}
\usepackage{amsthm}
\usepackage{float}
\usepackage{ulem}
\usepackage{xcolor}
\usepackage[utf8x]{inputenc}
\usepackage{comment}
\usepackage{framed}
\usepackage{caption}
\newcommand{\abs}[1]{\left\vert#1\right\vert}

\newcommand{\eps}{\varepsilon}

\DeclareMathAlphabet{\mathbbb}{U}{bbold}{m}{n}

\newtheorem{thm}{Theorem}[section]

\newtheorem{claim}[thm]{Claim}
\newtheorem{lem}[thm]{Lemma}
\newtheorem{prop}[thm]{Proposition}
\theoremstyle{definition}

\theoremstyle{remark}
\newtheorem{rem}[thm]{Remark}

\numberwithin{equation}{section}
\usepackage{graphicx}

\usepackage{enumerate}

\newcommand{\derive}[2]{\dfrac{\bd #1}{\bd#2}}

\newcommand{\twoarray}[2]{\begin{aligned}&{#1}\\&{#2}\end{aligned}}
\newcommand{\threearray}[3]{\begin{aligned}&{#1}\\&{#2}\\&{#3}\end{aligned}}

\newcommand{\matrice}{\begin{pmatrix}}
\newcommand{\ok}{\end{pmatrix}}

\newcommand{\scal}[2]{\langle{#1},{#2}\rangle}

\newcommand{\bd}{\partial}

\begin{document}
\title[]{On the isoperimetric inequality for the first positive Neumann eigenvalue on the sphere}
%\thanks{The first author acknowledges support of IDAM-GNAMPA.  The second author acknowledges support of INDAM-GNSAGA and of the the project ``Perturbation
%problems and asymptotics for elliptic differential equations: variational and potential theoretic methods'' funded by the MUR Progetti di Ricerca di Rilevante Interesse Nazionale (PRIN) Bando 2022 grant 2022SENJZ3.}

\author[Provenzano]{Luigi Provenzano}
\address{Dipartimento di Scienze di Base e Applicate per l'Ingegneria, Universit\`a di Roma ``La Sapienza'', Via Scarpa 12 - 00161 Roma, Italy, e-mail: {\sf luigi.provenzano@uniroma1.it}.}
\author[Savo]{Alessandro Savo}
\address{Dipartimento di Scienze di Base e Applicate per l'Ingegneria, Universit\`a di Roma ``La Sapienza'', Via Scarpa 12 - 00161 Roma, Italy, e-mail: {\sf alessandro.savo@uniroma1.it}.}

\begin{abstract}
We prove that geodesic disks are the unique maximizers of the first positive Neumann eigenvalue among all simply connected domains in $\mathbb S^2$ with prescribed area. We then show that geodesic disks are no longer maximisers for the class of spherical annuli.
\end{abstract}


\thanks{The authors would like to thank the Isaac Newton Institute for Mathematical Sciences, Cambridge, for support and hospitality during the programme Geometric spectral theory and applications, where part of the work on this paper was undertaken. This work was supported by EPSRC grant EP/Z000580/1. The authors acknowledge the support of the INdAM GNSAGA group. The first author acknowledges financial support from the project ``Perturbation problems and asymptotics for elliptic differential equations: variational and potential theoretic methods'' funded by the European Union – Next Generation EU and by MUR-PRIN-2022SENJZ3 and from the project ``Analisi Geometrica e Teoria  Spettrale su varietà Riemanniane ed Hermitiane'' of the INdAM GNSAGA}

\maketitle

\section{Introduction}

In this paper, we address the following question: does a spherical cap always maximize the second (i.e., first nontrivial) Neumann eigenvalue among all simply connected domains on the sphere with fixed area? Here we give a positive answer to this question.

\begin{thm}\label{main}
Let $\Omega$ be a simply connected domain of $\mathbb S^2$. Then
\begin{equation}\label{mainresult}
\mu_2(\Omega)\leq\mu_2(\Omega^{\star}),
\end{equation}
where $\Omega^{\star}$ is a geodesic disk with $|\Omega|=|\Omega^{\star}|$ and $\mu_2$ denotes the first positive Neumann eigenvalue. Equality holds if and only if $\Omega=\Omega^{\star}$.
\end{thm}

%For values of the area of $\Omega$ not exceeding $2\pi$ (half of the area of the sphere) %the inequality was proved by Bandle \cite{}

\smallskip

{\bf History of the problem.}  Isoperimetric inequalities of type \eqref{mainresult} are classical and have been studied since the times of Szeg\"o, around 1950. They are sometimes called Bandle-Szeg\"o-type inequalities. In \cite{szego_1} Szeg\"o proves  inequality \eqref{mainresult} when $\Omega$ is a  plane domain. Bandle, in the classical paper \cite{bandle2}, extends \eqref{mainresult} to simply connected Riemannian surfaces of area $A$ with Gaussian curvature bounded above by $K$, under the additional assumption $2\pi-KA\geq 0$. If $K\leq 0$ there is no restriction on $A$, while if $K>0$ the requirement is that $A\leq\frac{2\pi}{K}$. Note that a sphere of constant curvature $K$ has area $\frac{4\pi}{K}$. A consequence of the inequality of Bandle \cite{bandle2}  concerns spherical domains: the second Neumann eigenvalue of a spherical cap is maximal among all simply connected domains of $\mathbb S^2$ of fixed area $A$ {\it not exceeding} $2\pi$ (half the area of $\mathbb S^2$). The proof of \cite{bandle2} relies on conformal transplantation  in the spirit of Szeg\"o \cite{szego_1}. From \cite{bandle2} it follows that the inequality holds also for any simply connected domain of the hyperbolic plane $\mathbb H^2$.

\smallskip

Bandle's result  \cite{bandle2} on $\mathbb S^2$ was the best known until very recently. In \cite{lan_lau}, Langford and Laugesen were able to improve on  the restriction $A\leq 2\pi$ and go ``beyond the hemisphere'' by allowing values of the area satisfying $A\leq 4\pi c$, where $c=16/17\approx 0.941$. In the case of the sphere, this implies the isoperimetric inequality \eqref{mainresult} for simply connected domains of area up to about $94\%$ of the area of the sphere. The proof is a refinement of the approach of \cite{bandle2}, based on conformal transplantation, but contains important improvements and detours from the original proof. %Theorem \ref{main} was conjectured in \cite{lan_lau}.
 In the same paper, it was conjectured that \eqref{mainresult} would in fact hold for all simply connected spherical domains (see \cite[Open problems]{lan_lau}). This conjecture motivated the present work. 

\smallskip

A related question comes up naturally: % is the restriction to simply connected domains really necessary? 
 are there counterexamples if the boundary is not connected? 

\smallskip

 We remark that the isoperimetric inequality  \eqref{mainresult}   for domains of $\mathbb R^2$ and $\mathbb H^2$ holds also without restrictions on the topology, thanks to an argument due to Weinberger \cite{weinberger}, which is valid in the general non-simply connected case (and even in the arbitrary dimension $n$). 
The same proof establishes the inequality for all domains in  $\mathbb S^2$ contained in a hemisphere, see Ashbaugh-Benguria \cite{ash_beng}.  The assumption of being included in a hemisphere has been
weakened in Bucur, Martinet, Nahon \cite{bucur_sphere_2}.  In connection with these results, numerical investigations have been carried out in \cite{martinet_sphere}, providing numerical evidence and further insight in the structure of the problem. In the recent paper by Bucur, Laugesen, Martinet and Nahon \cite{bucur_gafa} the authors show that there exist multiply connected spherical domains (precisely, with five boundary components) that have  second eigenvalue strictly larger than that of the spherical cap with the same area. These examples consist of disks of radii $R>0.71   \pi$ with four small ellipses of large eccentricity removed. The elliptical holes are symmetrically
placed at the latitude of the hot spots of a second eigenfunction of the disk. So, the examples in \cite{bucur_gafa} have five boundary components. In \cite{bucur_gafa} the authors also describe another class of examples (``helmet-like domains''), providing numerical evidences and heuristic arguments for the failure of \eqref{mainresult}. These domains have four boundary components.  

\smallskip

At any rate, \eqref{mainresult} can't hold in general. One could then ask if there are counterexamples with a smaller number of boundary components, in particular: can \eqref{mainresult} be extended to spherical annuli (i.e., domains with only two boundary components)?  We actually tried to extend Theorem \ref{main} to that case, and as a first step we tested the validity of \eqref{mainresult}
on the family of rotationally symmetric annuli. A bit surprisingly, we found out that symmetric bands around the equator, with area sufficiently close to the area of $\mathbb S^2$, contradict \eqref{mainresult}.

%The assumption of being included in a  hemisphere has been weakened in Bucur, Martinet, Nahon \cite{bucur_sphere_2}: if the domain has area smaller than $|\mathbb S^n|/2$ and included in the complement of a spherical cap of the same area, inequality \eqref{mainresult} holds.All these results show that some additional conditions, such as the hemisphere condition in \cite{ash_beng}, or the conditions in \cite{bucur_sphere_2}, are necessary for the validity of the inequality on general domains.

 More precisely, we have the following theorem, which shows that the topological assumption in Theorem \ref{main} is the best possible.

\begin{thm}\label{thm_counter}
Let $\Omega_\eps$ be the annulus which in polar coordinates $(r,\theta)$ is described by $\Omega_\eps=\{(r,\theta):\eps<r<\pi-\eps\}$ and let $\Omega_\eps^\star$ be the geodesic disk of the same area. Then, as $\eps\to 0^+$
$$
\mu_2(\Omega_\eps)-\mu_2(\Omega_\eps^\star)=3\eps^4|\log\eps|+o(\eps^4|\log\eps|).
$$
\end{thm}
Therefore, for $\eps$ sufficiently small, the isoperimetric inequality (1.1) fails for the annulus $\Omega_\eps$. We stress that the failure of inequality \eqref{mainresult} occurs only when the area of the annulus $\Omega_\eps$ is very close to that of the sphere, that is, it must exceed $98.77\%$ of the total area of $\mathbb S^2$. See Figure 1 for a numerical study.
 Theorem \ref{thm_counter} is obtained by a careful examination of the asymptotics of the eigenfunctions which, after reduction to an ODE, are written in terms of Legendre functions.

% \medskip
 %\color{red} Una cosa del genere è possibile? For any integer $n$, there exist a set of $n$ points $K=\{x_1,\dots,x_n\}\in\mathbb S^2$ and $\eps=\eps(n)$
% such that the sphere minus the $\eps$-neighborhood of $K$ fails \eqref{mainresult}.  
\color{black}

%We also mention a strictly related result by Hersch \cite{hersch_sphere}, who proves that among all metrics of fixed area on $\mathbb S^2$, the round metric maximises the second eigenvalue (first nontrivial) of the Laplacian. The proof is again by conformal transplantation and follows the ideas of Szeg\"o \cite{szego_1}.

\subsection{Sketch of the proof of Theorem \ref{main}.}
 We recall that the Bandle-Szeg\"o-type inequality 
\eqref{mainresult}, valid for spherical domains of area at most $2\pi$, and its Langford-Laugesen improvement \cite{lan_lau}, which covers areas up to $\frac{16}{17}$ of the total area of the sphere, are proved by  {\it conformal transplantation}: this method consists in taking a conformal map from a simply connected surface $\Omega$ to the target optimal domain $\Omega^{\star}$ (a geodesic disk), and pulling back the Neumann eigenfunctions of the disk to $\Omega$ in order to use them as trial functions for the first nontrivial eigenvalue of $\Omega$. Note that there are many ways of choosing the conformal map, and this freedom guarantees the existence of a conformal map for which the pulled-back functions are orthogonal to the constants (hence they are valid test functions). A crucial requirement in the proof is that the radial profile of the eigenfunction of the geodesic disk of constant curvature is positive and increasing. This fails to be true when the area of the disk is large, and in \cite{lan_lau} the authors are able to relax this requirement of monotonicity replacing it by a monotonicity property for ratios of areas and integration by parts.

\smallskip

Our proof of  \eqref{mainresult} is new and does not use conformal transplantation of eigenfunctions. Rather, it employs the gauge invariance of the magnetic Laplacian and the level lines of the Green function. It can be sketched as follows (see Section \ref{sec_main} for complete details). 

\smallskip

{\it Step 1.} The first step is to introduce, for each point $p\in\Omega$, a Aharonov-Bohm magnetic potential $A_p$: this is a smooth $1$-form on $\Omega\setminus\{p\}$, which is closed, co-closed, and has flux $1$ around $\bd\Omega$ (and then around every loop enclosing $p$). This potential form gives rise to a magnetic Laplacian and a Neumann magnetic spectrum $\{\lambda_k(\Omega,A_p)\}_{k=1,2,\dots}$; by the well-known gauge invariance, since the fluxes of $A_p$ take only integer values, the Neumann spectrum and the Aharonov-Bohm spectrum with pole at $p$ are identical: for all $k=1,2,\dots$ one has
$
\mu_k(\Omega)=\lambda_k(\Omega,A_p)
$
and in particular:
$$
\mu_2(\Omega)=\lambda_2(\Omega,A_p).
$$
This happens for all poles $p\in\Omega$, and will give us freedom when handling the orthogonality relations. 

\smallskip

{\it Step 2.} Now, the test-functions. The magnetic potential $A_p$ is naturally expressed in terms of the Green function $\psi_p$ with pole at $p$ and Dirichlet boundary conditions, because
$$
A_p=-2\pi\star d\psi_p,
$$
where $\star$ is the Hodge-star operator (as a vector field $A_p=2\pi\nabla^{\perp}\psi_p$).  Thus, it is natural to isolate the class of $\psi_p$-radial functions, i.e., those functions which are real and constant on the level sets of $\psi_p$; restricting the Rayleigh quotient to this class of functions yields a Sturm-Liouville eigenvalue problem, whose lowest eigenvalue is positive, and is denoted $\kappa_1(\Omega,A_p)$. The isoperimetric inequality (together with the  Feynman-Hellmann formula) gives, for all $p\in\Omega$:
\begin{equation}\label{sketchone}
\kappa_1(\Omega,A_p)\leq \kappa_1(\Omega^{\star},A_{p^{\star}})
\end{equation}
where $p^{\star}$ is the center of the spherical cap $\Omega^{\star}$ having the same volume of $\Omega$. Direct inspection of the Aharonov-Bohm Laplacian of the pair $(\Omega^{\star},A_{p^{\star}})$ shows that there exists a radial  second Aharonov-Bohm eigenvalue of the geodesic disk (because the Green function of $\Omega^{\star}$ with pole at its center is in fact radial), hence
\begin{equation}\label{sketchtwo}
\kappa_1(\Omega^{\star},A_{p^{\star}})=\lambda_2(\Omega^{\star},A_{p^{\star}})=\mu_2(\Omega^{\star}),
\end{equation}
the second equality following again by gauge invariance. 

\smallskip

{\it Step 3.} It amounts to show that there is a point $\bar p\in\Omega$ such that 
\begin{equation}\label{sketchthree}
\lambda_2(\Omega,A_{\bar p})\leq \kappa_1(\Omega,A_{\bar p}).
\end{equation}
As $\mu_2(\Omega)=\lambda_2(\Omega,A_{\bar p})$, the Theorem follows from
\eqref{sketchone}, \eqref{sketchthree} and \eqref{sketchtwo}. The proof of \eqref{sketchthree} is obtained by mapping $\Omega$ conformally to the unit disk, and employing a fixed point argument to prove that the (radial) eigenfunction associated with $\kappa_1(\Omega,A_{p})$ is orthogonal to the eigenfunction associated with $\lambda_1(\Omega,A_p)=0$ for a suitable choice $\bar p$ of $p$. We refer to Section \ref{sec_barycenter} for complete details. 

\smallskip

We will present our main result, Theorem \ref{main}, and its proof, for spherical domains. However the very same proof (as in \cite{bandle2,lan_lau}) straightforwardly applies to simply connected, compact surfaces with boundary and Gaussian curvature bounded from above by $K$. Namely, we have that the second Neumann eigenvalue is largest when the domain is a geodesic disk of constant curvature $K$. 
 
\begin{thm}\label{main2}
Let $(\Omega,g)$ be a simply connected, compact Riemannian surface with boundary and Gaussian curvature bounded above by $K$. Assume that $4\pi-K|\Omega|_g\geq 0$, where $|\Omega|_g$ is the area of $(\Omega,g)$. Then
$$
\mu_2(\Omega,g)\leq\mu_2(\Omega^{\star}_K),
$$
where $\mu_2(\Omega,g)$ is the second Neumann eigenvalue of $(\Omega,g)$ and $\Omega^{\star}_K$ is a geodesic disk of constant curvature $K$ and area $\abs{\Omega}_g$.
 \end{thm}
This result extends \cite{bandle2} and \cite{lan_lau} with the best possible bound on $|\Omega|_g$.

\subsection{Final remarks} We stress the fact that our proof is not by conformal transplantation of eigenfunctions of the spherical cap on the domain $\Omega$. Our test functions are chosen to be, roughly speaking, the ``lowest energy functions'' that are constant on the level lines of the Green function of the domain; they are not necessarily transplantation of radial eigenfunctions of spherical caps and are more strictly related to the geometry of the domain itself.

\smallskip

The idea of using test functions derived from the Green function is inspired by a similar idea employed in \cite{CLPS25} to prove the reverse Faber-Krahn inequality for the first eigenvalue of the Neumann magnetic Laplacian with constant magnetic field $\beta>0$, in the weak magnetic field regime: in that case, the test-functions are taken in the class of functions which are real and constant on the level curves of the torsion function. 

\smallskip

 In \cite{MPS25}  the method of prescribed level lines (of the Green function) has been applied to prove several isoperimetric inequalities for the {\it first} eigenvalue of the Aharonov-Bohm Laplacian on surfaces, which is positive because non-integral fluxes are considered.
We are confident that this method could have other interesting applications in Spectral Geometry.

%\smallskip

%Finally, one could consider Aharonov-Bohm potentials with flux $\nu$ which is not an integer. In this case the {\it first} eigenvalue is positive and one can consider the maximization problem for this eigenvalue, which is of independent interest: see \cite{MPS25}. When the flux $\nu$ is not an integer,  the isoperimetric inequality for domains in $\mathbb H^2,\mathbb R^2$ and for simply connected domains of $\mathbb S^2$ with area not exceeding $2\pi$ has been established in \cite{CPS22}, using the classical techniques of Bandle, Szeg\"o and Weinberger. In \cite{MPS25}, for simply connected spherical domains the inequality holds without any restriction on the area, and the strategy of the proof is similar to the one described above.

\smallskip

{\bf Organization of the paper.} The paper is organised as follows: in Section \ref{sec_pre} we collect a few preliminaries on the magnetic Laplacian with closed potential $1$-form, the Green function of a surface and gauge invariance. In Section \ref{sec_radialspectrum} we introduce the notion of radial spectrum of the (magnetic) Laplacian: a spectrum obtained by restricting the eigenvalue problem to functions constant on the level lines of the Green function. We will find the good upper bound for the second Neumann eigenvalue by looking at this spectrum. In Section \ref{sec_main} we prove the main result, Theorem \ref{main}. It will be a consequence of three theorems encoding the main features of the proof, namely Theorems \ref{three}, \ref{four} and \ref{barycenter} that are stated in this section. These three theorems are proved in Sections \ref{sub_3}, \ref{sub_4} and \ref{sec_barycenter}, respectively. Theorem \ref{thm_counter} is proved in Section \ref{sec:counter}. At the end of the paper we have included an Appendix \ref{sec_app}, where, for the reader's convenience, we have collected a few details on some standard facts used in the proofs of the preceding sections, in order to keep the presentation self-contained.

\section{Preliminaries: magnetic Laplacian, Green function, and gauge invariance}\label{sec_pre}

Through this section,   $\Omega=(\Omega,g)$ is a bounded simply connected Riemannian surface with smooth boundary.

\subsection{Generalities on the magnetic Laplacian}
Let $p\in\Omega$ and let $A$ be a closed $1$-form in $\Omega\setminus \{p\}$. We denote by $\nu$ the {\it flux} of $A$, namely $\nu\doteq\frac{1}{2\pi}\oint_{c}A$, where $c$ is a simple, closed curve in $\Omega$ containing $p$ \footnote{travelled once in the counterclockwise direction; however the choice of the orientation does not affect our final result}. Let $d^A$ denote the magnetic differential: $d^Au=du-iuA$ (if $A=0$ it is the standard differential), and let $\delta^A$ its formal $L^2$-adjoint  (magnetic co-differential). Then $\delta F=-{\rm div}F$ if $F$ is a $1$-form.

The magnetic Laplacian is defined as $\Delta_A\doteq\delta^Ad^A$. If $A=0$ then $\Delta_A=\Delta$ is the usual Laplacian (the sign convention is that, in $\mathbb R^2$, $\Delta=-\partial^2_{xx}-\partial^2_{yy}$).  We consider the Neumann problem for the magnetic Laplacian
$$
\begin{cases}
\Delta_Au=\lambda u\,, & {\rm in\ }\Omega\setminus\{p\}\\
d^Au(N)=0\,, & {\rm on\ }\partial \Omega.
\end{cases}
$$
Here $N$ is the inner unit normal to $\partial \Omega$.  The spectrum is discrete, made of non-negative eigenvalues of finite multiplicity:
$$
0\leq\lambda_1(\Omega,A)\leq\lambda_2(\Omega,A)\leq\cdots\leq\lambda_k(\Omega,A)\leq\cdots\nearrow+\infty
$$
The eigenvalues are characterized by
\begin{equation}\label{minmaxA}
\lambda_k(\Omega,A)=\min_{{\substack{U\subset H^1_A(\Omega)\\{\rm dim}U=k}}}\max_{0\ne u\in U}\frac{\int_{\Omega}|d^Au|^2dv_g}{\int_{\Omega}|u|^2dv_g}.
\end{equation}
Here by $dv_g$ we denote the Riemannian volume form for the metric $g$. The Sobolev space $H^1_A(\Omega)$ is the space of (complex-valued) functions $u\in L^2(\Omega)$ such that $|d^Au|\in L^2(\Omega)$. % In general,  $H^1_A(\Omega)\subsetneq H^1(\Omega)$ (the space $H^1_A(\Omega)$ can be identified with $H^1(\Omega)$ after a unitary transformation if and only if $\nu\in\mathbb Z$).
We recall that a closed potential $1$-form is usually referred to as ``Aharonov-Bohm''-type potential. We refer to \cite{CPS22} for more details and a brief introduction to the spectral theory of Aharonov-Bohm magnetic Laplacians (see also \cite{FH_book}). 

When $A=0$ we have the usual Neumann problem for the Laplacian on $\Omega$:
$$
\begin{cases}
\Delta u=\mu u\,, & {\rm in\ }\Omega\\
du(N)=0\,, & {\rm on\ }\partial \Omega.
\end{cases}
$$
We use the letter $\mu$ for the usual Neumann eigenvalues:
$$
0=\mu_1(\Omega)<\mu_2(\Omega)\leq\cdots\leq\mu_k(\Omega)\leq\cdots\nearrow+\infty
$$
The Neumann eigenvalues are characterized by
$$
\mu_k(\Omega)=\min_{{\substack{U\subset H^1(\Omega)\\{\rm dim}U=k}}}\max_{0\ne u\in U}\frac{\int_{\Omega}|du|^2dv_g}{\int_{\Omega}|u|^2dv_g}.
$$

\subsection{The Green function} For $p\in\Omega$ let $\psi_p$ be the Green function with pole at $p$, unique solution of
\begin{equation}\label{green}
\begin{cases}
\Delta\psi_p=\delta_p\,, & {\rm in\ }\Omega\,,\\
\psi_p=0\,,& {\rm on\ }\partial\Omega,
\end{cases}
\end{equation}
where $\delta_p$ is the Dirac measure at $p$. Note that $\psi_p$ is positive, smooth and harmonic in $\Omega\setminus\{p\}$. By $\mathbb D$ we denote the unit disk in $\mathbb R^2$ centered at $0$. By $(r,\theta)$ we denote the usual polar coordinates in $\mathbb R^2$ based at $0$. We recall a few well-known facts on the Green function.

\begin{lem} We have:
\begin{enumerate}[i)]
\item The Green function of the unit disk $\mathbb D$ with pole at the origin is given by:
$$
\psi_0(r)=-\dfrac{1}{2\pi}\log r.
$$
\item The Green function is conformally invariant: if $\Phi:\Omega'\to\Omega$ is a conformal diffeomorphism, and $\psi_p$ is the Green function of $\Omega$ with pole at $p$, then $\psi_p\circ\Phi=\Phi^{\star}\psi_p$ is the Green function of $\Omega'$ with pole at $\Phi^{-1}(p)$.

\item $\psi_p$ has no critical points in $\overline\Omega\setminus\{p\}$: if $\Phi:\Omega\to \mathbb D$ is a conformal map, then $\psi_p=\Phi^{\star}\psi_0$, and $\psi_0$ has no critical points in $\overline{\mathbb D}\setminus\{0\}$.
\end{enumerate}
\end{lem}

We now consider, on $\Omega\setminus\{p\}$ the $1$-form:
\begin{equation}\label{ap}
A_p=-2\pi\star d\psi_p,
\end{equation}
where $\star$ is the Hodge-star operator associated with the metric $g$. The orientation is such that, if $e_1$ is the unit vector tangent to $\bd\Omega$, in the counterclockwise direction, then $\star e_1=N$, the inner unit normal.

\begin{lem} We have:
\begin{enumerate}[i)] 
\item The $1$-form $A_p$ is smooth, closed and co-closed, hence harmonic in $\Omega\setminus\{p\}$.
\item The flux of $A_p$ around $\bd\Omega$ (and around any loop enclosing $p$) is equal to $1$. 
\item  Let $(r,\theta)$ be the usual polar coordinates in $\mathbb R^2$ based at $0$. Then, on $\mathbb D\setminus\{0\}$ we have $A_0=d\theta$.
\end{enumerate}
\end{lem}
\begin{proof}
The proof is immediate. Since $\psi_p$ is harmonic in $\Omega\setminus\{p\}$ it follows by direct computation that $\star d\psi_p$ is closed and co-closed.

If $e_1$ is the unit vector tangent to $\bd\Omega$, in the counterclockwise direction, then $\star e_1=N$, hence
$$
\dfrac {1}{2\pi}\oint_{\bd\Omega}A_p=\int_{\bd\Omega}\star d\psi_p(e_1)ds_g=\int_{\bd\Omega}d\psi(N)ds_g=1.
$$ 
By $ds_g$ we denote the $1$-dimensional Riemannian measure for $g$. Point $iii)$ is a direct computation.
\end{proof}

\subsection{Gauge invariance} Fix a base point $x_0\in\Omega$, $x_0\ne p$ and define a function on $\Omega\setminus\{p\}$ as follows:
$$
\Theta_p(x):=\int_{c_x}A_p
$$
where $c_x$ is any curve joining $x_0$ to $x$. Since $A_p$ is closed, and the flux around a loop in $\Omega$ is an integer, one sees that choosing another such curve $c'_{x}$ one gets that 
$$
\int_{c_x}A_p-\int_{c_x'}A_p\in 2\pi\mathbb Z
$$
This means that the function 
$$
e^{i\Theta_p(x)}
$$
is well-defined and smooth in $\Omega\setminus\{p\}$. Moreover it belongs to $L^2(\Omega)$.

We observe that $e^{i\Theta_p}$ induces a linear isomorphism
$$
e^{i\Theta_p}:H^1(\Omega)\to H^1_{A_p}(\Omega),
$$
given by $u\mapsto e^{i\Theta_p}u$. This is a unitary operator:
$$
\int_{\Omega}|du|^2+|u|^2\,dv_g=\int_{\Omega}|d^{A_p}(e^{i\Theta_p}u)|^2+|e^{i\Theta_p}u|^2\, dv_g\,,\ \ \ \forall u\in H^1(\Omega),
$$
because one has the formula (gauge invariance):
$$
d^{A_p}(e^{i\Theta_p}u)=e^{i\Theta_p}du
$$
in $\Omega\setminus\{p\}$.
In particular
$$
\Delta_{A_p}=e^{i\Theta_p}\Delta e^{-i\Theta_p},
$$
so that $\Delta$ and $\Delta_{A_p}$ are unitarily equivalent. This proves the following.
\begin{lem}\label{one} For all $p\in\Omega$ and $k\in\mathbb N$:
$$
\lambda_k(\Omega,A_p)=\mu_k(\Omega).
$$
In particular, if $u$ is an eigenfunction of $\Delta_{A_p}$ then $e^{-i\Theta_p}u$ is an eigenfunction of $\Delta$, associated to the same eigenvalue. Vice versa, if $v$ is an eigenfunction of $\Delta$ then $e^{i\Theta_p}v$ is an eigenfunction of $\Delta_{A_p}$, associated to the same eigenvalue.
\end{lem}

%%%

\section{The radial spectrum at a point}\label{sec_radialspectrum}
Let $(\Omega,g)$ be a bounded, simply connected Riemannian surface. Let $M\doteq |\Omega|$ be the area of $(\Omega,g)$. Let $p\in\Omega$ and let $\psi_p$ be the Green function with pole at $p$ as in \eqref{green}. We introduce the function space:
$$
\mathcal R_p(\Omega)=\{u: u=g\circ\psi_p, \quad g\in H^1(0,\infty)\}\subset H^1_{A_p}(\Omega).
$$
Thus, $\mathcal R_p(\Omega)$ consists of all functions in $H^1_{A_p}(\Omega)$ which are constant on the level curves of $\psi_p$. We set:
\begin{equation}\label{K1-u}
\kappa_1(\Omega,A_p)=\min_{0\ne u\in\mathcal R_p(\Omega)}\frac{\int_{\Omega}|d^{A_p}u|^2dv_g}{\int_{\Omega}|u|^2dv_g}.
\end{equation}
We call $\kappa_1(\Omega,A_p)$ the first {\it radial eigenvalue} of $\Omega$ for the potential $A_p$. Note that $\kappa_1(\Omega,A_p)$ {\it does not need} to be an eigenvalue of $\Delta_{A_p}$ (hence of $\Delta$) in $\Omega$.  

Define a function $G_p:(0, M)\to (0,\infty)$ by:
\begin{equation}\label{Gp}
G_p(a)\doteq\int_{\psi_p=\beta_p(a)}\dfrac{1}{\abs{d\psi_p}}ds_g
\end{equation}
where $\beta_p(a)$ is such that the super level set $\{\psi_p>\beta_p(a)\}$ has area $a$.

We recall that $G_p(a)\sim 4\pi a$ as $a\to 0$. This follows from the fact that $\psi_p\sim-\frac{1}{2\pi}\log r$ as $r\to 0$, where $r$ is the geodesic distance from $p$, and that $\psi_p+\frac{1}{2\pi}\log r$ is a smooth function near $p$. In fact, as $a\to 0$, the behavior of $G_p$ is the same as that of $G_0$, defined as in \eqref{Gp} when $\Omega=\mathbb D$ and $p=0$, see Subsection \ref{sub:disk}.

The next result characterizes $\kappa_1(\Omega,A_p)$ as the minimizer of a one-dimensional Rayleigh quotient associated with a Sturm-Liouville problem.

\begin{lem}\label{two} We have the following:
\begin{enumerate}[i)]
\item 
\begin{equation}\label{K1}
\kappa_1(\Omega,A_p):=\min_{0\ne f\in\mathcal F_p}\frac{\int_0^{ M}\left(G_p(a)f'(a)^2+4\pi^2\frac{f(a)^2}{G_p(a)}\right)da}{\int_0^{M}f^2(a)da}.
\end{equation}
 where 
 $
\mathcal F_p=\{f\in L^2(0,M): \sqrt{G_p}f',f/\sqrt{G_p}\in L^2(0,M)\}.
$
\item $\kappa_1(\Omega,A_p)$ is the first eigenvalue of the  following Sturm-Liouville problem in $(0,M)$:
\begin{equation}\label{SL1}
\begin{cases}
-(G_pf')'+\frac{4\pi^2}{G_p}f=\kappa f\,, & {\rm in\ }(0,M)\,,\\
\lim_{a\to 0^+}G_p(a)f'(a)=f'(M)=0.
\end{cases}
\end{equation}
\end{enumerate}
\end{lem}

\begin{proof} We omit the subscript $p$ through the whole proof and simply write $\psi$ for $\psi_p$ and $G$ for $G_p$. Let $u=g\circ\psi$. One has:
$$
d^Au=(g'\circ\psi)d\psi-i(g\circ\psi)A;
$$
since $A$ and $d\psi$ are pointwise orthogonal by definition \eqref{ap}:
$$
\begin{aligned}
\abs{d^Au}^2&=(g'\circ\psi)^2\abs{d\psi}^2+(g\circ\psi)^2\abs{A}^2\\
&=\Big((g'\circ\psi)^2+4\pi^2(g\circ\psi)^2\Big)\abs{d\psi}^2.
\end{aligned}
$$
By the coarea formula
\begin{equation}\label{u-g}
\begin{aligned}
\int_{\Omega}\abs{d^Au}^2&=\int_0^{\infty}\Big(g'(t)^2+4\pi^2g(t)^2\Big)\int_{\psi=t}\abs{d\psi}\,ds_g\,dt\\
&=\int_0^{\infty}\Big(g'(t)^2+4\pi^2g(t)^2\Big)dt.
\end{aligned}
\end{equation}
The last equality follows from
$$
\int_{\psi=t}\abs{d\psi}\,ds_g=\int_{\psi=t}\derive{\psi}{N}\,ds_g=\int_{\{\psi>t\}}\Delta\psi\, dv_g=1.
$$
We now change variable as follows. Write:
$$
\alpha(t)\doteq\abs{\{\psi>t\}}=\int_t^{\infty}\int_{\psi=s}\dfrac{1}{\abs{d\psi}}\,ds_g\,dt
$$
so that
$$
\alpha'(t)=-\int_{\psi=t}\dfrac{1}{\abs{d\psi}}\,ds_g.
$$
Since $\psi$ has no critical points in $\overline\Omega\setminus\{p\}$, then $\abs{d\psi}\geq c>0$. This implies that $\alpha: (0,\infty)\to (0,M)$ is smooth, strictly decreasing, and admits a smooth inverse $\beta:(0,M)\to (0,\infty)$. We set:
$$
t=\beta(a), \quad g\circ\beta=f,
$$
hence $g(t)=f(a)$. Since $\beta(\alpha(t))=t$ we have
$
\beta'(\alpha(t))\alpha'(t)=1, 
$
which means
$$
\beta'(a)=\dfrac{1}{\alpha'(\beta (a))}=-\dfrac{1}{\int_{\psi=\beta(a)}\frac{1}{\abs{d\psi}}\,ds_g}.
$$
Defining the function $G:(0,M)\to(0,\infty)$ as in \eqref{Gp} by
\begin{equation*}
G(a)=\int_{\psi=\beta(a)}\frac{1}{\abs{d\psi}}\,ds_g,
\end{equation*}
we conclude that 
$$
\beta'(a)=-\dfrac{1}{G(a)}.
$$
Now 
$$
\threearray
{g'(t)=g'(\beta(a))=(g\circ\beta)'(a)\cdot\dfrac{1}{\beta'(a)}=\dfrac{f'(a)}{\beta'(a)}=-f'(a)G(a),}
{g(t)=f(a),}
{dt=\beta'(a)da=-\dfrac{da}{G(a)}.}
$$
We conclude
$$
\begin{aligned}
\int_{\Omega}\abs{d^Au}^2\,dv_g&=\int_0^{\infty}\Big(g'(t)^2+4\pi^2g(t)^2\Big)dt\\
&=\int_0^{ M}\left(G(a)f'(a)^2+4\pi^2\frac{f(a)^2}{G(a)}\right)da.
\end{aligned}
$$
On the other hand, if $u=g\circ\psi$:
$$
\begin{aligned}
\int_{\Omega}u^2\,dv_g&=\int_0^{\infty}g(t)^2\int_{\psi=t}\dfrac{1}{\abs{d\psi}}\,ds_g\,dt\\
&=\int_0^Mf(a)^2\,da.
\end{aligned}
$$
This proves $i)$. The proof of $ii)$ is standard Sturm-Liouville theory and follows directly from $i)$. We sketch the main facts. Assume that $f$ is a minimizer of the Rayleigh quotient in \eqref{K1}.  Then, taking $f+t\phi$ in the Rayleigh quotient, with $\phi\in C^{\infty}_c(0,M)$,  deriving with respect to $t$ and exploiting the minimality of $f$, we get the differential equation in \eqref{SL1} solved by $f$ in $(0,M)$. As for the boundary conditions, consider again the Rayleigh quotient with test functions $f+t\phi$. First, take $\phi$ supported in a neighborhood of $M$. Since $G(M)>0$, integrating by parts and using the fact that $f$ satisfies the differential equation in the interior, we get the usual Neumann condition $f'(M)=0$. Take now $\phi$ supported in a neighborhood of $0$. We have that $G(a)\sim 4\pi a$ as $a\to 0$, so we have a singular endpoint (of Bessel type) and we get the condition $\lim_{a\to 0^+}G(a)f'(a)=0$.
\end{proof}

We remark that we have defined $\kappa_1(\Omega,A_p)$ as the minimum of a Rayleigh quotient over a certain subspace of $H^1_{A_p}(\Omega)$. Then $\kappa_1(\Omega,A_p)$ turns out to be the first eigenvalue of a suitable Sturm-Liouville problem \eqref{SL1}. It is clear that we can define a whole sequence of eigenvalues $\{\kappa_k(\Omega,A_p)\}_{k=1}^{\infty}$ via the min-max procedure, which then coincides with the spectrum of \eqref{SL1}. We call this spectrum the {\it radial spectrum} at $p$. Note that these need not to be actual eigenvalues of $\Delta_{A_p}$ in $\Omega$. However, as we shall see, sometimes they are. In this paper we just work with $\kappa_1(\Omega,A_p)$.

\subsection{The case of the unit disk}\label{sub:disk}

On the unit disk $\mathbb D$ the Green function with pole at $0$ is
$$
\psi_0(r)=-\dfrac{1}{2\pi}\log r
$$
hence $\{\psi_0=t\}$ is the circle of radius $r=e^{-2\pi t}$, so that
$$
\abs{\{\psi_0=t\}}=2\pi e^{-2\pi t}, \quad \alpha(t)=\abs{\{\psi_0>t\}}=\pi e^{-4\pi t}.
$$
Hence, if $\beta(a)=\alpha^{-1}(a)$, inverting $a=\pi e^{-4\pi t}$ gives
$$
\beta(a)=-\dfrac{1}{4\pi}\log(\frac a{\pi}).
$$
Now $d\psi_0=-\frac{1}{2\pi r}dr$, hence 
$\abs{d\psi_0}=\dfrac{1}{2\pi r}=\dfrac{1}{2\pi}e^{2\pi t}$.
Finally 
$$
\gamma_0(t)\doteq\int_{\psi_0=t}\dfrac{1}{|d\psi_0|}ds=\abs{\{\psi_0=t\}}\cdot\dfrac{1}{\abs {d\psi_0}}=4\pi^2e^{-4\pi t},
$$
where $ds$ is the arc-length element. Hence, 
$$
G_0(a)=\gamma_0(\beta(a))=4\pi a.
$$
\subsection{The case of a spherical cap}\label{sub:sphericalcup}
Let now $\Omega^{\star}\subset\mathbb S^2$ be a spherical cap of radius $R$ with center $p^{\star}$. In polar coordinates $(r,\theta)$, where $r$ is the geodesic distance from $p^{\star}$, the Green function with pole $p^{\star}$ is
$$
\psi_{p^{\star}}(r)=-\frac{1}{2\pi}\log\left(\frac{\tan(r/2)}{\tan(R/2)}\right).
$$
Proceeding as in the case of the disk, we see that $G_{p^{\star}}(a)=a(4\pi-a)$. However, the same fact will also drop from the following isoperimetric inequality, which will be needed later. 
Let $G_p:(0,M)\to (0,\infty)$ be defined as in \eqref{Gp} in terms of the Green function of $\Omega$ with pole $p$. Then we have:

\begin{lem}\label{lem31} One has $G_{p^\star}\leq G_p$ with equality a.e. if and only if $\Omega=\Omega^{\star}$ and $p=p^{\star}$. Moreover $G_{p^{\star}}(a)=a(4\pi-a)$ for all $a\in (0,M)$.
\end{lem}
\begin{proof}
We omit the subscript $p$ through the whole proof. For all $t\in (0,\infty)$ we have
\begin{equation}\label{pre-isoperimetric}
\begin{aligned}
\abs{\{\psi=t\}}&=\int_{\psi=t}1\,ds\\
&=\int_{\psi=t}\abs{d\psi}^{\frac 12}\cdot\dfrac{1}{\abs{d\psi}^{\frac 12}}\,ds\\
&\leq \Big(\int_{\psi=t}\abs{d\psi}\,ds\Big)^{\frac 12}\cdot 
\Big(\int_{\psi=t}\dfrac{1}{\abs{d\psi}}\,ds\Big)^{\frac 12}\\
&=\Big(\int_{\psi=t}\dfrac{1}{\abs{d\psi}}\,ds\Big)^{\frac 12}
\end{aligned}
\end{equation}
with equality if and only if $\abs{d\psi}$ is constant on $\psi=t$. Recall that $\alpha(t)\doteq |\{\psi>t\}|$ and $\beta$ is its inverse. Taking $t=\beta(a)$ and passing to the variable $a$, \eqref{pre-isoperimetric} reads
$$
\abs{\{\psi=\beta(a)\}}^2\leq G(a).
$$
We use the well-known geometric isoperimetric inequality for spherical domains: $L^2\geq A(4\pi-A)$, where $A$ is the area and $L$ is the boundary length. We refer e.g., to \cite{chavel_isop} for a proof. Moreover, 
$\abs{\{\psi>\beta(a)\}}=a$. Therefore
$$
\abs{\{\psi=\beta(a)\}}^2\geq a(4\pi-a).
$$
We conclude that
$$
G(a)\geq a(4\pi-a)
$$
with equality if and only if each level set $\{\psi=t\}$ is a circle and $\abs{d\psi}$ is constant on $\psi=t$. This last condition implies that the level sets are parallel to one another, hence $\psi$ must be a radial function; hence $\Omega^{\star}=\Omega$ and $p=p^{\star}$. Finally, for a spherical cap we have equality everywhere, in particular
$G_{p^{\star}}(a)=a(4\pi-a)$ for all $a\in (0,M)$.

\end{proof}
From now on we will set
$$
G^{\star}(a)\doteq G_{p^{\star}}(a)=a(4\pi-a).
$$
%%%

\section{Proof of Theorem \ref{main}}\label{sec_main}
 
 From now on we assume that $\Omega$ is a smooth, simply connected domain of the round sphere $\mathbb S^2$. Let $\Omega^{\star}$  be the spherical cap with the same volume of $\Omega$. Let
 $$
M\doteq\abs{\Omega^{\star}}=\abs{\Omega}.
 $$
Theorem \ref{main} will follow by combining Gauge invariance (Lemma \ref{one}) with three results that we state in this section, namely Theorems \ref{three}, \ref{four} and \ref{barycenter}.

\medskip

The first result is a comparison between the first radial eigenvalue $\kappa_1(\Omega,A_p)$ of $\Omega$ (with any potential $A_p$) and the first radial eigenvalue $\kappa_1(\Omega^{\ast},A_{p^{\ast}})$ of $\Omega^{\star}$ with potential $A_{p^{\star}}$, where $p^{\star}$ is the center of $\Omega^{\star}$.

\begin{thm}\label{three} Let $p^{\star}$ be the center of $\Omega^{\star}$. Then
$$
\kappa_1(\Omega,A_p)\leq \kappa_1(\Omega^{\star}, A_{p^{\star}})
$$
for all $p\in\Omega$, with equality if and only if $\Omega=\Omega^{\star}$ and $p=p^{\star}$.
\end{thm}

\medskip

We recall that a radial eigenvalue $\kappa_1(\Omega,A_p)$ is not necessarily a Neumann eigenvalue of $\Delta_{A_p}$. The next result states that the first radial eigenvalue for $\Omega^{\star}$ with pole $p^{\star}$ is actually a Neumann eigenvalue of $\Delta_{A_{p^{\star}}}$ (hence of $\Delta$), more precisely, it is the second eigenvalue.

\begin{thm}\label{four} One has $\kappa_1(\Omega^{\star}, A_{p^{\star}})=\lambda_2(\Omega^{\star},A_{p^{\star}})$.
%
%This implies, from Theorem \ref{three}, that
%$$
%\kappa_1(\Omega,A_p)\leq %\lambda_2(\Omega^{\star},A_{p^{\%star}})
%$$
%for all $p\in\Omega$.
\end{thm}

A radial eigenfunction, as just said, does not need to be an eigenfunction of $\Delta_{A_p}$. Moreover it does not need to be orthogonal to the first eigenfunction $e^{i\Theta_p}$ of $\Delta_{A_p}$, so that it cannot be used to estimate from above $\lambda_2(\Omega,A_p)$ using the min-max principle \eqref{minmaxA}. A center of mass argument shows that for some $\bar p$ a first radial eigenfunction is indeed orthogonal to $e^{i\Theta_{\bar p}}$, implying the following result which is needed to conclude the proof of Theorem \ref{main}. 

\begin{thm}\label{barycenter} There exists $\bar p\in\Omega$ such that
$$
\lambda_2(\Omega, A_{\bar p})\leq \kappa_1(\Omega,A_{\bar p}).
$$
\end{thm}

The proof of Theorem \ref{main} is achieved by combining Lemma \ref{one} and Theorems \ref{three}, \ref{four} and \ref{barycenter}:
$$
\begin{aligned}
\mu_2(\Omega)&=\lambda_2(\Omega,A_{\bar p})\quad\text{by Lemma \ref{one}}\\
&\leq \kappa_1(\Omega,A_{\bar p})\quad\text{by Theorem \ref{barycenter}}\\
&\leq \kappa_1(\Omega^{\star},A_{p_{\star}})\quad\text{by Theorem \ref{three}}\\
&=\lambda_2(\Omega^{\star},A_{p_{\star}})\quad\text{by Theorem \ref{four}}\\
&=\mu_2(\Omega^{\star})\quad\text{by Lemma \ref{one}}.
\end{aligned}
$$
Moreover, if $\mu_2(\Omega)=\mu_2(\Omega^\star)$, then $\kappa_1(\Omega,A_{\bar p})=\kappa_1(\Omega^{\star},A_{p^{\star}})$, and from Theorem \ref{three} we see that $\Omega=\Omega^{\star}$ and $p=p^{\star}$.

%%%

\section{Proof of Theorem \ref{three}}\label{sub_3}

The proof of Theorem \ref{three} is a consequence of the geometric isoperimetric inequality (Lemma \ref{lem31}), and a monotonicity result for the first eigenvalue of a weighted Sturm-Liouville problem with respect to the weight. 
\medskip

Let $G_p:(0,M)\to (0,\infty)$ be defined as in \eqref{Gp} and let $G^{\star}= a(4\pi-a)$. 
Let now $G:(0,M)\to (0,\infty)$ be a smooth, positive function, and $\kappa_1(G)$ be defined by
\begin{equation}\label{RayleghG}
\kappa_1(G):=\min_{0\ne f\in\mathcal F}\frac{\int_0^{ M}\left(G(a)f'(a)^2+4\pi^2\frac{f(a)^2}{G(a)}\right)da}{\int_0^{M}f^2(a)da},
\end{equation}
where $\mathcal F=\{f\in L^2(0,M): \sqrt{G}f',f/\sqrt{G}\in L^2(0,M)\}$ (see also Lemma \ref{two}, $i)$).  We also assume that $G$ satisfies
\begin{equation}\label{asy}
G(a)\sim 4\pi a, \quad a\to 0.
\end{equation}
Note that $\kappa_1(\Omega,A_p)=\kappa_1(G_p)$. We have the following.

\begin{lem}\label{32} If $G\geq G^\star$ on $(0,M)$ then
$$
\kappa_1(G^\star)\geq \kappa_1(G).
$$
If $G>G^\star$ on a set of positive measure, the inequality is strict.
\end{lem}
\begin{proof}
For $t\in [0,1]$ define $G_t:(0,M)\to (0,\infty)$ by
$$
G_t=(1-t)G^\star+tG.
$$
and let $\kappa_1(G_t)$ be the minimizer in \eqref{RayleghG} with $G=G_t$. As in the proof of Lemma \ref{two}, $\kappa_1(G_t)$ is the first eigenvalue of the Sturm-Liouville problem \eqref{SL1} (with $G_t$ replacing $G_p$); as $\kappa_1(G_t)$ is simple, we can apply the Feynman-Hellmann formula (see \cite[VII-\S4, p. 408, formula 4.56]{kato1}) which in our case gives
$$
\dfrac{d}{dt}\kappa_1(G_t)=\int_0^M\Big(G_t(a)^2f'_t(a)^2-4\pi^2f_t(a)^2\Big)\dfrac{G(a)-G^\star(a)}{G_t(a)^2}\,da
$$
where $f_t$ is the first positive eigenfunction associated to the eigenvalue $\kappa_1(G_t)$, normalized so that $\int_0^Mf_t^2=1$.
It will then be enough to show that $G_t^2f_t'^2-4\pi^2 f_t^2<0$ on $(0,M)$.  Equivalently, let us fix an arbitrary $t\in [0,1]$ and consider the smooth function $R:(0,M)\to\mathbb R$ defined by
$$
R=\dfrac{G_tf'_t}{f_t}.
$$
Since $f_t'(M)=0$ and $f_t(M)>0$, we see
\begin{equation}\label{rmzero}
R(M)=0.
\end{equation}
It is enough to show that $\abs{R(a)}<2\pi$ for all $a\in (0,M)$. We differentiate $R$ and get:
\begin{equation}\label{rprime}
R'=\dfrac{4\pi^2-R^2}{G_t}-\kappa_1(G_t)
\end{equation}
which implies that if $\abs{R}\geq 2\pi$ then $R'<0$. 

\medskip

{\bf First case.} Assume  that there exists $a_0\in (0,M)$ such that $R(a_0)\geq 2\pi$.

\smallskip

Then $R'(a_0)<0$ and one sees easily that $R$ is decreasing (and positive) on $(0,a_0)$. Therefore
 $\lim_{a\to 0}R(a)$ exists. We claim that, necessarily:
\begin{equation}\label{limit}
\lim_{a\to 0}R(a)=+\infty.
\end{equation}
In fact, since $R$ is decreasing on $(0,a_0)$ and $R(a_0)\geq 2\pi$, we see that $R-2\pi$ is uniformly bounded below by a positive constant on $(0,a_0/2)$ and, on that interval, there exists $c^2>0$ such that $4\pi^2-R^2\leq -c^2$. Integrating \eqref{rprime} on $(a,a_0/2)$ we see:
$$
\begin{aligned}
R(a_0/2)-R(a)&=\int_a^{a_0/2}\Big(\dfrac{4\pi^2-R^2(x)}{G_t(x)}-\kappa_1(G_t)\Big)\,dx\\
&<-c^2\int_{a}^{a_0/2}\dfrac{dx}{G_t(x)}.
\end{aligned}
$$
As $a\to 0$ we know that  $G_t(a)\sim 4\pi a$; hence the left-hand side diverges to $-\infty$ which proves \eqref{limit}.

\smallskip

Since $\lim_{a\to 0}R(a)=+\infty$ we get, from \eqref{rprime} (since $G_t(a)\sim 4\pi a$ when $a\to 0$):
$$
\lim_{a\to 0}\dfrac{aR'(a)}{R^2(a)}=-\dfrac{1}{4\pi}, \quad\text{hence}\quad \lim_{a\to 0}a\Big(\dfrac{1}{R}\Big)'(a)=
\dfrac{1}{4\pi}.
$$
Then there exists $\bar a>0$ such that, for  $x\in (0,\bar a)$: 
$$
\Big(\dfrac{1}{R}\Big)'(x)\geq\dfrac{1}{8\pi s}.
$$
Integrating the inequality for $x\in (a,\bar a)$ we obtain
$$
\dfrac{1}{R(\bar a)}-\dfrac{1}{R(a)}\geq \dfrac{1}{8\pi}\log(\frac{\bar a}{a}).
$$
Taking the limit as $a\to 0$ on both sides we reach a contradiction with \eqref{limit}. Therefore
$$
R(a)<2\pi,
$$
for all $a\in (0,M)$. 

\smallskip

{\bf Second case.} Assume that there exists $a_0$ such that $R(a_0)\leq -2\pi$. 

\smallskip

Then $R'(a_0)<0$ and $R'<0$ on $(a_0,M)$. Therefore $R(M)<-2\pi$ which is a contradiction with \eqref{rmzero}. Hence
$$
R(a)>-2\pi
$$
for all $a\in (0,M)$.
The proof is complete.
\end{proof}

Theorem \ref{three} now follows by taking $G=G_p$ and recalling by Lemma \ref{lem31} that $G_p\geq G^\star$.

\section{Proof of Theorem \ref{four}}\label{sub_4}

Let $\Omega^{\star}=B(0,R)$ be the geodesic disk (spherical cap) of radius $R\in (0,\pi)$ in $\mathbb S^2$, centered at $p^{\star}$. We briefly recall the description of the spectrum of the Neumann Laplacian on $\Omega^{\star}$. We use polar coordinates $(r,\theta)$ around $p^{\star}$
 and separate variables. As usual, we find a basis of $L^2(\Omega^{\star})$ of eigenfunctions in the form $u(r,\theta)=v(r)e^{ik\theta}$, where $k\in\mathbb Z$ (see e.g., \cite[\S II.5]{chavel}).  Expressing the  Laplacian $\Delta$ in polar coordinates, we see that $u$ as above is an eigenfunction of the Laplacian on $\Omega^{\ast}$ satisfying the Neumann boundary condition on $\partial\Omega^{\ast}$ if and only if the radial part $v(r)$ is an eigenfunction of the following Sturm-Liouville problem:
\begin{equation}\label{SLK}
SL_k:
\begin{cases}
-v''-\cot r \,v'+\dfrac{k^2}{\sin^2r}v=\mu v\,, & {\rm in\ }(0,R)\,,\\
\lim_{r\to 0^+}r v'(r)=v'(R)=0.
\end{cases}
\end{equation}
 For each $k\in\mathbb Z$,  problem $SL_k$ has a countable set of eigenvalues, denoted 
$$
\mu_{k1}(R)\leq \mu_{k2}(R)\leq\cdots\leq\mu_{kj}(R)\leq\cdots\nearrow+\infty
$$
with associated eigenfunctions $v_{kj}(r)$, where $j=1,2,\dots$, $k\in\mathbb Z$. Note that $\mu_{kj}=\mu_{-kj}$ for all $k\in\mathbb Z$, hence we may confine the analysis of the eigenvalues to $k=0,1,...$.
The smallest eigenvalue is $\mu_1(\Omega^{\star})=\mu_{01}(R)=0$ with eigenspace given by the constants. The second eigenvalue (lowest positive eigenvalue) is denoted $\mu_2(\Omega^{\star})$: it could be either $\mu_{02}(R)$ (radial) or  $\mu_{11}(R)$ (phase equal to $1$). In \cite[Proposition 6.1]{lang_laug_robin} it is shown that
\begin{equation}\label{second_N} \mu_{11}(R)<\mu_{02}(R)\,,\ \ \ \forall R\in(0,\pi).
\end{equation}

Therefore $\mu_2(\Omega^{\star})=\mu_{11}(R)$. In particular,  the  corresponding eigenspace is spanned by $e^{i\theta}v_{11}(r)$ and $e^{-i\theta}v_{11}(r)$. 
Using gauge invariance (Lemma \ref{one}), we see that $\lambda_2(\Omega^{\star},A_{p^{\star}})=\mu_2(\Omega^{\ast})=\mu_{11}(R)$ has multiplicity $2$; since $A_{p^{\star}}=d\theta$ and $e^{i\Theta_{p^{\star}}}=e^{i\theta}$,  the corresponding eigenspace is  spanned by $v_{11}(r)$ and $e^{2i\theta}v_{11}(r)$.

In particular, we see that $\lambda_2(\Omega^{\star},A_{p^{\star}})$ admits an eigenfunction which is real and radial (i.e., constant on the level lines of $\psi_{p^{\star}}$), and this is $v_{11}$; since $v_{11}$ does not change sign, it must be a first radial eigenfunction of $\Delta_{A_{p^{\star}}}$. In conclusion
$$
\kappa_1(\Omega^{\star},A_{p^{\star}})=\lambda_2(\Omega^{\star},A_{p^{\star}})
$$
as asserted.

\section{Proof of Theorem \ref{barycenter}}\label{sec_barycenter}

Recall that we have to prove that there exists $\bar p\in\Omega$ such that 
$$
\lambda_2(\Omega,A_{\bar p})\leq \kappa_1(\Omega,A_{\bar p})
$$
where on the right-hand side we have the lowest eigenvalue of the radial spectrum associated to the pair $(\Omega,A_{\bar p})$. To achieve that, 
consider, for each $p\in\Omega$, a unit norm eigenfunction $u_p$ (watch: of the radial eigenvalue problem \eqref{K1-u}) associated to  $\kappa_1(\Omega,A_{ p})$ so that
$$
\frac{\int_{\Omega}|d^{A_{p}}u_p|^2dv_S}{\int_{\Omega}|u_p|^2dv_S}=\kappa_1(\Omega,A_{ p}).
$$
Here $dv_S$ is the volume element of the standard round metric $g_S$ on $\mathbb S^2$. Recall that $\lambda_1(\Omega,A_p)=0$ with associated eigenfunction $e^{i\Theta_p}$; hence, if we can manage to find $p\in\Omega$ such that $\int_{\Omega}u_{p}e^{-i\Theta_{p}}\,dv_S=0$ then one can use $u_p$ as a test-function for the second eigenvalue and conclude that: 
$$
\lambda_2(\Omega,A_{ p})\leq \frac{\int_{\Omega}|d^{A_{p}}u_p|^2dv_S}{\int_{\Omega}|u_p|^2dv_S}=\kappa_1(\Omega,A_{ p}).
$$
Note that since $u_p$ is real, $\int_{\Omega}u_pe^{-i\Theta_p}dv_S=0$ if and only if $\int_{\Omega}u_pe^{i\Theta_p}dv_S=0$, and we will use this equivalent form for the orthogonality (this is not relevant for the proof, and we made this choice to simplify the presentation).

The proof of Theorem \ref{barycenter} is then reduced to the proof of the following claim:

\smallskip

\begin{claim}\label{lemmaW} The map $W:\Omega\to\mathbb C$ defined by
\begin{equation*}%\label{lemmaW} 
W(p)=\int_{\Omega}u_pe^{i\Theta_p}dv_S
\end{equation*}
has a zero.
\end{claim}

\medskip

\subsection{First step: apply the Uniformization Theorem.} 
It will be convenient to see $W$ as a function on the unit disk: this is done by mapping $\Omega$ conformally onto the unit disk. We then apply Brouwer fixed point theorem to get the result.

\smallskip
Let us fix once and for all a reference conformal map
$$
\Phi:\mathbb D\to\Omega
$$
so that $(\Omega,g_S)$ is isometric to $(\mathbb D,\rho^2 g_E)$ for a certain conformal factor $\rho^2$, where $g_S$ is the round metric on $\mathbb S^2$ and $g_E$ is the Euclidean metric on $\mathbb D$ (the Uniformization Theorem guarantees the existence of such a map). Therefore $\Phi^\star g_S=\rho^2g_E$. We will denote the volume elements of $g_S$ and $g_E$ as $dv_S$ and $dv_E$, respectively. Note that $\rho$ is smooth and positive on $\mathbb D$. Next, we use the conformal group of the unit disk: for $q\in\mathbb D$, we consider the M\"obius map $M_q:\mathbb D\to\mathbb D$ (which is a conformal automorphism of the disk):
$$
M_q(z)=\dfrac{z-q}{1-\bar qz}\quad\text{with inverse}\quad M_q^{-1}(w)=
\dfrac{w+q}{1+\bar q w}.
$$

\begin{lem}\label{Phiq} The following facts hold:
\begin{enumerate}[i)]
\item Let $p\in\Omega$ and $q=\Phi^{-1}(p)$. Then the map
$$
\Phi_q\doteq\Phi\circ M_q^{-1}:\mathbb D\to\Omega
$$
is conformal and sends the origin to $p$: $\Phi_q(0)=p$.
\item $(\Omega,g_S)$ is isometric to $(\mathbb D,\rho_q^2g_E)$, where $\rho_q^2$ is the conformal factor of $\Phi_q^*g_S$. Explicitly,
\begin{equation}\label{conf_q}
\rho_q^2(z)=\rho^2\left(\frac{z+q}{1+\bar q z}\right)\frac{(1-|q|^2)^2}{|1+\bar q z|^4}
\end{equation}
and in particular
\begin{equation}\label{conf_q_2}
\rho_q^2\, dv_E=(M_q^{-1})^\star(\rho^2\,dv_E)
\end{equation}
\end{enumerate}
\end{lem}

\begin{proof}
Point $i)$ is immediate by observing that $M_q^{-1}(0)=q$.  Point $ii)$ follows from a straightforward calculation. In fact
$$
\begin{aligned}
\Phi^{\star}_qg_S&=(M_q^{-1})^{\star} \Phi^{\star}g_S\\
&=(M_q^{-1})^{\star}\rho^2g_E\\
&=(\rho^2\circ M_q^{-1})(M_q^{-1})^{\star}g_E
\end{aligned}
$$
and 
$$
(M_q^{-1})^{\star}g_E=\abs{(M_q^{-1})'}^2g_E=\frac{(1-|q|^2)^2}{|1+\bar q z|^4}g_E.
$$
\end{proof}
Since  conformal maps preserve the Green function, we see that the pull-back by $\Phi_q$ will take the Green function of $\Omega$ at $p$ to the Green function of $\mathbb D$ at the origin, which is explicit and is denoted by $\psi_0$; recall that in polar coordinates 
$$
\psi_0(r)=-\dfrac{1}{2\pi}\log r.
$$
Likewise, the pull-back of the potential one-form $A_p$ will be $d\theta$, independently on $q$.
We summarize these facts in the following lemma.

\begin{lem} \label{Phiqtwo} Fix $p\in\Omega$. Let $q\in\mathbb D$ such that $p=\Phi(q)$, and let $\Phi_q$ as in Lemma \ref{Phiq}. Then
\begin{enumerate}[i)]
\item $\Phi_q^{\star}\psi_p=-\dfrac{1}{2\pi}\log r$.
\item $\Phi_q^{\star}A_p=d\theta$.
\item $\Phi_q^{\star}e^{i\Theta_p}=e^{i\theta}$.
\end{enumerate}
\end{lem}

Recall that $u_p= f_p\circ\alpha_p\circ\psi_p$ is a minimizer of \eqref{K1-u}, i.e., a first radial eigenfunction associated to 
$\kappa_1(\Omega,A_p)$. Recall also  that $f_p: (0,M)\to \mathbb R$ is a first eigenfunction of the corresponding one-dimensional problem \eqref{SL1} and $\alpha_p(t)\doteq|\{\psi_p>t\}|$. We take $f_p$ normalized by $\int_0^Mf_p^2(a)da=1$ and positive: this corresponds to $\int_{\Omega}u_p^2\,dv_S=1$, $u_p>0$. 

\smallskip

The orthogonality relation $W(p)=0$ (where  $W:\Omega\to\mathbb C$ is as in Claim \ref{lemmaW}), viewed in $\mathbb  D$, becomes the following.

\begin{lem}\label{vq} Write $p=\Phi(q)$. Then
\begin{enumerate}[i)]
\item If $v_q:\mathbb D\to\mathbb C$ is $v_q\doteq \Phi_q^{\star}u_p$ then 
$
W(p)=\int_{\mathbb D}v_qe^{i\theta}\rho^2_qdv_E.
$
\item The function $v_q$ is radial, and in fact 
$$
v_q=f_p\circ s_q
$$
where $s_q(r)=\int_{B(0,r)}\rho^2_qdv_E$ is the area of the disk of radius $r$ in the conformal metric $\rho^2_qg_E$.

\end{enumerate}
\end{lem}

%With this at hand, Claim \ref{lemma W}

\begin{proof} Identity $i)$ is simply a conformal change of coordinates through $\Phi_q$, using Lemma \ref{Phiq} and Lemma \ref{Phiqtwo} 
$$
\begin{aligned}
W(p)&=\int_{\Omega}u_pe^{i\Theta_p}dv_S\\
&=\int_{\mathbb D}\Phi^{\star}_qu_p\Phi^{\star}_q(e^{i\Theta_p})\Phi^{\star}_q(dv_S)\\
&=\int_{\mathbb D}v_qe^{i\theta}\rho^2_qdv_E.
\end{aligned}
$$
Proof of $ii)$. Since $\psi_p\circ\Phi_q=-\dfrac{1}{2\pi}\log r$, the level line $\{\psi_p=t\}$ is mapped to $r=e^{-2\pi t}$ and then
$$
\alpha_p(t)=\abs{\psi_p>t}=\abs{r<e^{-2\pi t}}_{\rho^2_q g_E}=s_q(e^{-2\pi t}).
$$
Now:
$$
\begin{aligned}
\Phi_q^{\star}u_p&=f_p\circ\alpha_p\circ\psi_p\circ\Phi_q\\
&=f_p\circ\alpha_p\Big(-\frac{1}{2\pi}\log r\Big)\\
&=f_p\circ s_q(r)
\end{aligned}
$$
as asserted.
\end{proof}

With this at hand, we note that Claim \ref{lemmaW} becomes

\smallskip

\begin{claim}\label{lemmaV}   The function $V:\mathbb D\to \mathbb C$ defined by
\begin{equation*}%\label{lemmaV}
V(q)=\int_{\mathbb D}v_qe^{i\theta}\rho^2_q dv_E
\end{equation*}
has a zero in $\mathbb D$, where $v_q$ is as in Lemma \ref{vq}, $i)$.
\end{claim}

\medskip

In the next steps we will prove the following fact. 

%The proof of Lemma \ref{lemmaV}, and hence of Theorem \ref{barycenter}, is complete once we show the following fact.

\begin{thm}\label{final} Let $V(q)=\int_{\mathbb D}v_qe^{i\theta}\rho^2_q\,dv_E$. Then:
\begin{enumerate}[i)]
\item $V$ is continuous in $\mathbb D$.
\item $V$ extends to a continuous function on $\overline{\mathbb D}$, and 
$
V(q)=-\sqrt Mq
$
for all $q\in\bd\mathbb D$.
\end{enumerate}
\end{thm} 
If $V$ is viewed as a vector field on $\overline{\mathbb D}$, then $V$ is continuous and points inward at every point of the boundary. Then, it must have a zero in $\mathbb D$: this is an easy consequence of Brouwer fixed point theorem (we have included a proof in Appendix \ref{fixpoint}). This proves Claim \ref{lemmaV} and, with it, Theorem \ref{barycenter}.  

\medskip

\subsection{\bf Second step: change of variables}

Let $\Delta_{A_0}$ be the Aharonov-Bohm Laplacian with potential $A_0=-2\pi\star d\psi_0$, where $\psi_0$ is the Green function of the unit disk with pole at the origin. 
In order to prove Theorem \ref{final} we interpret $v_q$ as an eigenfunction of $\Delta_{A_0}$ with a density that depends on $q$, and we study the continuity in $q$ of the eigenfunction.

\begin{lem}\label{radial_pb} The function $v=v_q(r)$ is the first radial eigenfunction of the weighted problem:
\begin{equation}\label{weightedeigenvalue}
\begin{cases}
\Delta_{A_0}v=\mu \tilde\rho^2_qv\,, & {\rm in\ } \mathbb D\,,\\
d^{A_0}v(N)=0\,, & {\rm on\ } \bd\mathbb D
\end{cases}
\end{equation}
where the weight $\tilde\rho^2_q$ is radial, and equals
$$
\tilde\rho^2_q(r)=\dfrac{1}{2\pi}\int_0^{2\pi}\rho^2_q(r,\theta)\,d\theta
$$
for all $r\in (0,1)$.
\end{lem} 
The proof  of Lemma \ref{radial_pb} consists in a change of variables and it is rather standard. We pass from the variable $a\in (0,M)$ (recall that $f_p$ is an eigenfunction of the Sturm-Liouville problem \eqref{SL1} in $(0,M)$) to the variable $r\in(0,1)$, where $r$ is the radius of a disk of area $a$ for the metric $\rho_q^2g_E$. For the reader's convenience, we have included the details of the change of variables in Appendix \ref{app_weighted}.

\smallskip
By Gauge invariance (Lemma \ref{one}) we deduce from Lemma \ref{radial_pb} that $w_q\doteq v_q(r)e^{-i\theta}$ is a Neumann eigenfunction for the Laplacian with weight:
\begin{equation}\label{Neumann_density}
\begin{cases}
\Delta w=\mu\tilde\rho_q^2 w\,, & {\rm in\ }\mathbb D\,,\\
dw(N)=0\,, & {\rm on\ }\partial\mathbb D.
\end{cases}
\end{equation}
In particular, $w_q=v_q(r)e^{-i\theta}$ is the first eigenfunction with angular part $e^{-i\theta}$ (recall that $v_q(r)>0$).

\subsection{Third step: the weight concentrates at the boundary}

The weight $\tilde\rho^2_q(r)$ is in fact obtained by averaging $\rho_q^2(r,\theta)$ over the circle of radius $r$.

\smallskip

The main fact for us is that when $q\to\bd\mathbb D$ the weight tends to concentrate at the boundary, in the following precise sense. Recall that $M$ is the volume of $\Omega^{\star}$, that is
$
M=\int_{\mathbb D}\rho^2_q\,dv_E=\int_{\mathbb D}\tilde\rho^2_q\,dv_E
$
for all $q\in\mathbb D$. In Appendix \ref{app_neumanntosteklov} we will prove the following Lemma.

\begin{lem}\label{lem_radial} For any $p>1$ and any $u\in W^{1,p}(\mathbb D)$ we have
\begin{equation}\label{condition_convergence}
\left|\int_{ \mathbb D}\tilde\rho_q^2u\,dv_E-\frac{M}{2\pi}\int_{\partial \mathbb D}u\,ds\right|\leq\omega_p(|q|)\|u\|_{W^{1,p}( \mathbb D)},
\end{equation}
where $\omega_p(|q|)\to 0$ as $|q|\to 1$. 
\end{lem}

 Lemma \ref{lem_radial} is stating that, if a sequence of points $q_n\in\mathbb D$ converges to the boundary point $e^{i\gamma}$ and if we set:
$$
\mu_n\doteq \tilde\rho_{q_n}^2 dv_E, \quad \mu\doteq \frac{M}{2\pi}ds
$$
where $ds$ is the arc-length element, then, as $n\to\infty$
\begin{equation}\label{mun}
\mu_n\to\mu
\end{equation}
in $W^{1,p}(\mathbb D)^*$ for all $p>1$. 

\medskip

 We are then studying the behavior of the Neumann eigenvalues and eigenfunctions of the Laplacian on a disk with a radial density that concentrates at the boundary keeping the mass fixed. This phenomenon of mass concentration to the boundary has been studied in \cite{lambertiprovenzano1} (see also \cite{dallaprovenzano,lambertiprovenzano2}). In \cite{lambertiprovenzano1} it has been proved that the Neumann problem with density of fixed mass concentrating uniformly at the boundary of $\mathbb D$ is well-behaved at the limit and converges to the Steklov problem on $\mathbb D$. In the case at hand we will see that the normalized eigenfunction $w_q=v_q(r)e^{-i\theta}$  tends to a second normalized eigenfunction of the Steklov problem
\begin{equation}\label{Steklov_limit}
\begin{cases}
\Delta u=0\,, & {\rm in\ }\mathbb D\,,\\
du(N)=\dfrac{M}{2\pi}\sigma u\,, & {\rm on\ }\partial\mathbb D
\end{cases}
\end{equation}
as $q\to\partial\mathbb D$. 

\smallskip
A more comprehensive and general analysis of how eigenvalues and eigenfunctions depend on measures is presented in \cite{GKL}. There, it is shown that the concentration \eqref{mun} guarantees convergence of spectra and convergence of eigenfunctions in $H^1(\mathbb D)$. 
We state the following proposition, which is a special case of \cite[Propositions 4.8 and 4.11]{GKL}, adapted to our simpler situation.
\begin{prop}\label{propGKN}
Let $\{q_n\}_{n=1}^{\infty}\subset\mathbb D$ be a sequence of points in $\mathbb D$ such that $q_n\to e^{i\gamma}\in\partial\mathbb D$. Suppose that
$$
\left|\int_{\mathbb D}\tilde\rho_{q_n}^2uv\,dv_E-\frac{M}{2\pi}\int_{\partial\mathbb D}uv\,ds\right|\leq\omega(|q_n|)\|u\|_{H^1(\mathbb D)}\|v\|_{H^1(\mathbb D)}
$$
for all $u,v\in H^1(\mathbb D)$, where $\omega(|q|)\to 0$ as $|q|\to 1$ is some modulus of continuity not depending on $u,v$. Let $\{\mu_k^{(q_n)}\}_{k=1}^{\infty}$ denote the eigenvalues of \eqref{Neumann_density} with $q=q_n$, and let $\{u_k^{(q_n)}\}_{k=1}^{\infty}$ be an orthonormal basis of $L^2(\mathbb D,\tilde\rho_{q_n}^2dv_E)$ of corresponding eigenfunctions. Let $\{\sigma_k\}_{k=1}^{\infty}$ denote the eigenvalues of \eqref{Steklov_limit} and let $\{u_k\}_{k=1}^{\infty}$ be an orthonormal basis of $L^2(\partial\mathbb D,\frac{M}{2\pi}ds)$ of corresponding eigenfunctions. Then
$$
\lim_{q_n\to e^{i\gamma}}\mu_k^{(q_n)}=\sigma_k
$$
and, up to extracting a subsequence,
$$
\lim_{q_n\to e^{i\gamma}}\|u_k^{(q_n)}-u_k\|_{H^1(\mathbb D)}=0.
$$
The convergence is along the whole sequence if $\sigma_k$ is simple.
\end{prop}
\begin{rem}Note that $\mu_k^{(q)}=\kappa_k(\Omega,A_p)$ with $\Phi(q)=p$, hence, when the pole of the Green function approaches the boundary, the radial spectrum converges to the spectrum of the Steklov problem \eqref{Steklov_limit}.
\end{rem}

\subsection{Fourth step: proof of Theorem \ref{final}}

\begin{proof}[Proof of Theorem \ref{final}]
\begin{enumerate}[i)]
 \item The map $V$ is continuous from $\mathbb D$ to $\mathbb C$. In fact, as soon as $q\in\mathbb D$, $v_q$ and $\rho_q^2$ vary smoothly with $q$. In particular, $q\mapsto v_q$ is continuous in $C^0([0,1]])$.
 \item It is not restrictive to consider sequences $\{q_n\}_{n=1}^{\infty}$ such that $q_n\to e^{i\gamma}\in\partial\mathbb D$. We are in the hypothesis of Propositions \ref{propGKN}  (see also \cite[Proposition 4.8 and 4.11]{GKL} and \cite[\S 5.1 and \S 5.2]{GKL}): in fact, we have, for any $u,v\in H^1(\mathbb D)$ and $p>1$ by Lemma \ref{lem_radial} that 
 $$
\left|\int_{ \mathbb D}\tilde\rho_{q_n}^2uv\,dv_E-\frac{M}{2\pi}\int_{\partial \mathbb D}uv\,ds\right|\leq\omega_p(|q_n|)\|uv\|_{W^{1,p}( \mathbb D)},
 $$
then, choosing $1<p<2$, we have by Sobolev Embedding
$$
\|uv\|_{W^{1,p}(\mathbb D)}\leq C\|u\|_{H^1(\mathbb D)}\|v\|_{H^1(\mathbb D)}.
$$
Proposition \ref{propGKN} gives convergence in $H^1(\mathbb D)$ of the eigenfunctions up to extracting subsequences, unless the limiting eigenvalue is simple, which is essentially the case at hand,  because we are looking at a specific sequence: $\{v_{q_n}(r)e^{i\theta}\}_{n=1}^{\infty}$, where the angular part is fixed along the whole sequence.

The eigenfunctions and eigenvalues of the limiting Steklov problem \eqref{Steklov_limit} are well-known: $\sigma_1=0$, $\sigma_2=\sigma_3=\frac{2\pi}{M}$, etc. An $L^2(\partial\mathbb D,\frac{M}{2\pi}ds)$-orthonormal basis of the eigenspace corresponding to $\sigma_2=\sigma_3$  is given by $\left\{\frac{r}{\sqrt{M}}e^{i\theta},\frac{r}{\sqrt{M}}e^{-i\theta}\right\}$. The eigenspace corresponding to the zero eigenvalue is one-dimensional and spanned by constant functions. All other eigenvalues are double. Now, for $|q_n|$ close to $1$, $\mu_2^{(q_n)}=\mu_3^{(q_n)}$ and this eigenvalue is double, converging to $\sigma_2=\sigma_3$. An associated orthonormal basis of eigenfunctions is then of the following form: $\{\tilde v_{q_n}(r)e^{i\theta},\tilde v_{q_n}(r)e^{-i\theta}\}$, for some $\tilde v_{q_n}(r)>0$ and this forces $\tilde v_{q_n}(r)=v_{q_n}(r)$: in fact, by a change of variables, defining $\tilde f_{p_n}:(0,M)\to\mathbb R$ by $\tilde v_{q_n}=\tilde f_{p_n}\circ s_{q_n}$, we have that $\tilde f_{p_n}$ must satisfy \eqref{SL1}, hence $\tilde f_{p_n}=f_{p_n}$ due to the normalization and the choice of the sign. Recall that $p_n=\Phi(q_n)$. Then, up to extracting a subsequence, $v_{q_n}(r)e^{i\theta}$ converges in $H^1(\mathbb D)$ to $\frac{r}{\sqrt{M}}e^{i\theta}$ as $q_n\to e^{i\gamma}$. The fact that we have fixed the angular part guarantees that the convergence is along the whole sequence. Finally, we have convergence in $C^0(\overline{\mathbb D})$, because for all $r\in (0,1)$: 
\begin{equation}\label{czerohone}
\left|v_{q_n}(r)e^{-i\theta}-\frac{r}{\sqrt{M}}e^{-i\theta}\right|^2\leq 
\frac{1}{2\pi}\|v_{q_n}(r)e^{-i\theta}-\frac{r}{\sqrt{M}}e^{-i\theta}\|_{H^1(\mathbb D)}^2.
\end{equation}
To verify that, set for simplicity of notation $\phi_n(r)\doteq v_{q_n}(r)-\frac{r}{\sqrt{M}}$.\\
Since $v_{q_n}(0)=0$ for all $q_n$, we have $\phi_n(0)=0$ for all $n$. Then, for all $r\in (0,1)$:

\begin{equation}\label{czeroconvergence}
\begin{aligned}
\abs{\phi_n(r)e^{-i\theta}}^2&=2\int_0^r\phi'_n(y)\phi_n(y)dy\\
&\leq \int_0^r\Big(\phi'_n(y)^2y+\phi_n(y)^2\frac 1y\Big)dy\\
&\leq \int_0^1\Big(\phi'_n(r)^2r+\phi_n(r)^2\frac 1r\Big)dr\\
&=\dfrac1{2\pi}\int_{\mathbb D}\abs{\nabla(\phi_n(r)e^{-i\theta})}^2dv_E\\
&\leq \frac{1}{2\pi}\|\phi_n(r)e^{-i\theta}\|_{H^1(\mathbb D)}^2
\end{aligned}
\end{equation}
which proves \eqref{czerohone}.
Therefore $v_{q_n}(r)$ converges uniformly to $r/\sqrt{M}$ as $n\to\infty$ and in particular
 $$
\lim_{q_n\to e^{i\gamma}}v_{q_n}(1)=\frac{1}{\sqrt{M}}.
 $$
Let
$$
F_{q_n}=v_{q_n}e^{i\theta},
$$
so that
$$
V(q_n)=\int_{\mathbb D}F_{q_n}(w)\rho^2_{q_n}(w)\,dv_E(w).
$$
From Lemma \ref{Phiq}, $ii)$ we know
$$
\rho^2_{q_n}\,dv_E=(M_{q_n}^{-1})^\star(\rho^2\, dv_E)
$$
and then,  by changing variables $w=M_{q_n}(z)$, 
$$
V({q_n})=\int_{\mathbb D}F_{q_n}(M_{q_n}(z))\rho^2(z)dv_E(z).
$$
Now, for all $z\in\mathbb D$
$$
\lim_{q_n\to e^{i\gamma}}M_{q_n}(z)=-e^{i\gamma}
$$
hence, as ${q_n}\to e^{i\gamma}$ we see that $F_{q_n}(M_{q_n}(z))\to -\frac{1}{\sqrt{M}}e^{i\gamma}$ for all $z$ and $V({q_n})\to -\sqrt{M}e^{i\gamma}$. Thus $V$ extends continuously to $\bd{\mathbb D}$, and for all $q\in \bd{\mathbb D}$ one has
$$
V(q)=-\sqrt{M}q.
$$
\end{enumerate}
\end{proof}

\section{Proof of Theorem \ref{thm_counter}}\label{sec:counter}
On the sphere $\mathbb S^2$ with polar coordinates
$(r,\theta)\in (0,\pi)\times [0,2\pi)$ we consider the domain
$$
\Omega_{\eps}\doteq \{(r,\theta): \eps<r<\pi-\eps\}
$$
for $\eps\in(0,\pi/2)$. Note that $\Omega_{\eps}$ is the sphere with two antipodal disks of radius $\eps$ removed. 
We denote by $D_{\eps}$ the geodesic disk of radius $\pi-\eps$:
$$
D_{\eps}\doteq \{(r,\theta): 0<r<\pi-\eps\}
$$
which is the sphere with a disk of radius $\eps$ removed. Theorem \ref{thm_counter} is a consequence of the following:

\begin{thm}\label{mainasy} As $\eps \to 0^{+}$ one has:
$$
\twoarray
{\mu_{2}(D_{\eps})=2-\frac 32\eps^{2}-\frac 32\eps^{4}|\log\eps|+o(\eps^{4}|\log\eps|)}
{\mu_{2}(\Omega_{\eps})=2-3\eps^{2}-3\eps^{4}|\log\eps|+o(\eps^{4}|\log\eps|)}
$$
\end{thm}
Theorem \ref{thm_counter} becomes a corollary of Theorem \ref{mainasy}.
\begin{proof}[Proof of Theorem \ref{thm_counter}] Let $s(\eps)$ be such that $\Omega^{\star}_{\eps}=D_{s(\eps)}$. 
Then the condition $\abs{\Omega_{\eps}}=\abs{D_{s(\eps)}}$ gives $\cos(s(\eps))=2\cos\eps-1$
hence
$$
s(\eps)=\arccos(2\cos\eps-1)=2\arccos(\sqrt{\cos\eps}).
$$
Expanding near $\eps=0$ we have
$
s(\eps)=\sqrt 2\eps+O(\eps^{3}).
$
From Theorem \ref{mainasy} we see that
$$
\mu_{2}(\Omega^{\star}_{\eps})=\mu_{2}(D_{s(\eps)})=2-3\eps^{2}-6\eps^{4}|\log\eps|+o(\eps^{4}|\log\eps|)
$$
as $\eps\to 0^+$, hence
$$
\mu_{2}(\Omega_{\eps})-\mu_{2}(\Omega^{\star}_{\eps})=3\eps^{4}|\log\eps|+o(\eps^{4}|\log\eps|)
$$
as $\eps\to 0^+$ and the assertion follows. 
\end{proof}
The rest of this section is devoted to the proof of Theorem \ref{mainasy}.

\subsection{Well-known facts on the first positive Neumann eigenvalue of $\Omega_\eps$ and $D_\eps$} Here we identify the first positive Neumann eigenvalue of $\Omega_\eps$ and $D_\eps$ as the first eigenvalue of suitable Sturm-Liouville problems.

As usual, we find a basis of $L^2(D_\eps)$ of eigenfunctions in the form $u(r,\theta)=w(r)e^{ik\theta}$, where $k\in\mathbb Z$ (see e.g., \cite[\S II.5]{chavel}).  Expressing the  Laplacian $\Delta$ in polar coordinates, we see that $u$ as above is an eigenfunction of the Laplacian on $D_\eps$ satisfying the Neumann boundary condition  if and only if the radial part $w(r)$ is an eigenfunction of the following Sturm-Liouville problem:
\begin{equation}\label{ODE_disk_k}
\begin{cases}
-w''(r)-\cot(r)w'(r)+\frac{k^2}{\sin^2(r)}w(r)=\eta w(r)\,, & {\rm in\ }(0,\pi-\eps)\,,\\
\lim_{r\to 0^+}\sin(r)w'(r)=w'(\pi-\eps)=0.
\end{cases}
\end{equation}
From standard Sturm-Liouville theory, when $k=0$ the boundary condition is equivalent to requiring $w'(0)=0$, while for $k\ne 0$ it is equivalent to $w(0)=0$.

 For each $k\in\mathbb Z$,  problem \eqref{ODE_disk_k} has a countable set of eigenvalues, denoted 
$$
\eta_{k1}(\eps)\leq \eta_{k2}(\eps)\leq\cdots\leq\eta_{kj}(\eps)\leq\cdots\nearrow+\infty
$$
with associated eigenfunctions $w_{kj,\eps}(r)$, where $j=1,2,\dots$, $k\in\mathbb Z$. Note that $\eta_{kj}(\eps)=\eta_{-kj}(\eps)$ for all $k\in\mathbb Z$, hence we may confine the analysis of the eigenvalues to $k=0,1,...$. 
The smallest eigenvalue is $\mu_1(D_\eps)=\eta_{01}(\eps)=0$ for all $\eps\in(0,\pi)$,  with eigenspace given by the constants. The second eigenvalue (lowest positive eigenvalue) is denoted $\mu_2(D_\eps)$: it could be either $\eta_{02}(\eps)$ (radial) or  $\eta_{11}(\eps)$.

Analogously, there exists a basis of $L^2(\Omega_\eps)$ of eigenfunctions in the form $u(r,\theta)=v(r)e^{ik\theta}$, where $k\in\mathbb Z$. Then $u$ is an eigenfunction of the Laplacian on $\Omega_\eps$ satisfying the Neumann boundary conditions on the two boundary circles if and only if the radial part $v(r)$ is an eigenfunction of the following Sturm-Liouville problem:
\begin{equation}\label{ODE_annulus_k}
\begin{cases}
-v''(r)-\cot(r)v'(r)+\frac{k^2}{\sin^2(r)}v(r)=\sigma v(r)\,, & {\rm in\ }(\eps,\pi-\eps)\,,\\
v'(\eps)=v'(\pi-\eps)=0.
\end{cases}
\end{equation}
This is a regular Sturm-Liouville problem. For each $k\in\mathbb Z$,  problem \eqref{ODE_annulus_k} has a countable set of eigenvalues, denoted 
$$
\sigma_{k1}(\eps)\leq \sigma_{k2}(\eps)\leq\cdots\leq\sigma_{kj}(\eps)\leq\cdots\nearrow+\infty
$$
with associated eigenfunctions $v_{kj,\eps}(r)$, where $j=1,2,\dots$, $k\in\mathbb Z$. Also here we may confine the analysis of the eigenvalues to $k=0,1,...$. The second eigenvalue (lowest positive eigenvalue) is denoted $\mu_2(\Omega_\eps)$: as before, it could be either $\sigma_{02}(\eps)$ (radial) or  $\sigma_{11}(\eps)$ (see Appendix \ref{app:lemma_radial_eig}).

We are now in position to identify the first positive Neumann eigenvalue of $D_\eps$ and $\Omega_\eps$.

\begin{lem}\label{lem_first_annulus}We have:
\begin{enumerate}[i)]
\item For all $\eps\in(0,\pi)$
$$
\mu_2(D_\eps)=\eta_{11}(\eps)
$$
and a corresponding eigenspace is spanned by $\{w_{11,\eps}(r)e^{i\theta},w_{11,\eps}(r)e^{-i\theta}\}$.
\item For all $\eps\in(0,\pi/2)$
$$
\mu_2(\Omega_\eps)=\sigma_{11}(\eps)
$$
and a corresponding eigenspace is spanned by $\{v_{11,\eps}(r)e^{i\theta},v_{11,\eps}(r)e^{-i\theta}\}$.
\end{enumerate}
\end{lem}
Point $i)$ is proved e.g., in \cite{lang_laug_robin} and point $ii)$ can be proved using the same argument. For the reader's convenience, we have included a proof of $ii$) in Appendix \ref{app:lemma_radial_eig}.

\smallskip

We will also need the following preliminary estimate:

\begin{lem}\label{lem_est_eigen} There exist $C>0$ and $\eps_0>0$ such that, for all $\eps\in(0,\eps_0)$
\begin{enumerate}[i)]
\item  $|\mu_2(D_\eps)-2|\leq C\eps^2$ and $\mu_2(D_\eps)< 2$.
\item $|\mu_2(\Omega_\eps)-2|\leq C\eps^2$ and $\mu_2(\Omega_\eps)<2$.
\end{enumerate}
The constant $C$ is independent of $\eps$.
\end{lem}

Lemma \ref{lem_est_eigen} tells us that at the limit as $\eps\to 0^+$, $\mu_2(\Omega_\eps)$ and $\mu_2(D_\eps)$ converge to $\mu_2(\mathbb S^2)=2$ from below. It also gives the rate of convergence: $\eps^2$.  The rate of convergence in Lemma \ref{lem_est_eigen} is again well-known and holds when small holes of more general shapes are removed from a spherical domain. We refer e.g., to \cite[Proposition 3]{bucur_gafa}  for a proof. The fact that the eigenvalues converge from below follows from a direct computation: it is enough to use $\sin(r)e^{i\theta}$ as test function in the Rayleigh quotients of $\mu_2(D_\eps)$ and $\mu_2(\Omega_\eps)$.

\smallskip

\subsection{Generalities on Legendre functions}  Consider  the ODE on $(0,\pi)$:
\begin{equation}\label{ODE}
-v''(r)-\cot(r)\,v'(r)+\dfrac{1}{\sin^{2}r}v(r)=\lambda v(r).
\end{equation}
We have seen that $\eta_{11}(\eps)$ is the first eigenvalue of \eqref{ODE} on the subinterval $(0,\pi-\eps)$ while $\sigma_{11}(\eps)$ is the first eigenvalue of \eqref{ODE} on the subinterval $(\eps,\pi-\eps)$, both with Neumann conditions.

\smallskip

Set $x=\cos r$ and $v(r)=u(x)=u(\cos r)$. Then $u:(-1,1)\to\mathbb R$ satisfies
\begin{equation}\label{ODEtwo}
(1-x^{2})u''(x)-2xu'(x)+\Big(\lambda-\dfrac{1}{1-x^{2}}\Big)u(x)=0.
\end{equation}
Let $\nu>0$ be so that
$$
\lambda=\nu(\nu+1).
$$
Consider  the ODE on $(-1,1)$:
\begin{equation}\label{legendre}
(1-x^{2})u''(x)-2xu'(x)+\nu(\nu+1)u(x)=0
\end{equation}
By Frobenius theory, since the coefficients of \eqref{legendre} are analytic on $(-1,1)$, so is any  solution of it. There is a unique solution, denoted $P_{\nu}(x)$, which is analytic at $x=1$ and satisfies $P_{\nu}(1)=1$. More precisely we have, as $x\to 1^{-}$ (see \cite[14.3.4 and 15.2.1]{nist}): 
\begin{equation}\label{asyP}
P_{\nu}(x)=1-\dfrac{\nu(\nu+1)}{2}(1-x)+O(1-x)^{2}.
\end{equation}
By Frobenius method, one can construct a second independent solution, denoted $Q_{\nu}(x)$, which is such that, when $\nu\notin \mathbb Z$ (see \cite[14.9.10]{nist}):
\begin{equation}\label{explicitQ}
Q_{\nu}(x)=\dfrac{\pi}{2\sin(\pi\nu)}\left(\cos(\pi\nu)P_{\nu}(x)-P_{\nu}(-x)\right).
\end{equation}
We will need the following refined asymptotics as $x\to 1^{-}$ of $Q_\nu$, which we prove in Appendix \ref{app_Q}:
\begin{lem}\label{appthree}The function $Q_{\nu}(x)$ has a logarithmic singularity at $x=1^-$, and, as $x\to 1^{-}$:
\begin{equation}\label{asyQ}
Q_{\nu}(x)=-\frac 12\log(1-x)+c(\nu)+\dfrac{\nu(\nu+1)}{4}(1-x)\log(1-x)+o((1-x)|\log(1-x)|),
\end{equation}
where $c(\nu)=\frac{1}{2}\log(2)-\gamma-\psi(\nu+1)$. Here $\gamma$ is Euler's constant and $\psi(z)\doteq\frac{\Gamma'(z)}{\Gamma(z)}$.
\end{lem}

The functions $P_{\nu}(x)$ and $Q_{\nu}(x)$ are known in the literature as the {\it standard Legendre functions of degree $\nu$} of the first and second kind, or {\it Ferres} functions (since we are in the interval $(-1,1)$), and constitute the standard basis of solutions of \eqref{legendre} in $(-1,1)$. We can write any solution of \eqref{ODEtwo} in terms of $P_\nu(x),Q_\nu(x)$, or, equivalently, $P_\nu(x),P_\nu(-x)$:

\begin{lem}
The space of solutions in $(-1,1)$ of \eqref{ODEtwo} is two-dimensional, spanned by:
$$
U_{1}(x)=\sqrt{1-x^{2}}P'_{\nu}(x), \quad U_{2}(x)=\sqrt{1-x^{2}}P'_{\nu}(-x)
$$
\end{lem}
The proof is a straightforward calculation. An immediate consequence is that:
\begin{lem} Two linearly independent solutions of \eqref{ODE} are given by
$$
p_{\nu}(r)=\sin (r)P'_{\nu}(\cos r), \quad q_{\nu}(r)=p_{\nu}(\pi-r)
$$
so that 
$$
\lim_{r\to 0^{+}}p_{\nu}(r)=0, \quad \lim_{r\to \pi^{-}}q_{\nu}(r)=0,
$$
and then $p_{\nu}$ is regular at $r=0$ and $q_{\nu}$ is regular at $r=\pi$.
\end{lem}

Let $D_{\eps}$ be the geodesic ball of radius $\pi-\eps$. Then we write $\eta_{11}(\eps)=\nu_{\eps}(\nu_{\eps}+1)$. Since $p_{\nu_{\eps}}$ is regular at $0$, in order to be a  Neumann eigenfunction of $D_{\eps}$ associated to $\eta_{11}(\eps)$, one imposes the boundary condition
$
p'_{\nu_{\eps}}(\pi-\eps)=0.
$
We summarize:
\begin{lem} A Neumann eigenfunction of $D_{\eps}$ associated to $\eta_{11}(\eps)=\nu_{\eps}(\nu_{\eps}+1)$ is defined on $(0,\pi-\eps)$ and given by  
$$
v(r)\doteq p_{\nu_{\eps}}(r)
$$
where $\nu_{\eps}>0$ is the lowest positive root of 
$$
p'_{\nu_{\eps}}(\pi-\eps)=0.
$$
\end{lem}

We now turn our attention to the domain $\Omega_{\eps}$. By symmetry, the eigenfunction $v_{11,\eps}$ must be even or odd with respect to $\pi/2$, but since it is the first eigenfunction of \eqref{ODE_annulus_k} with $k=1$ it doesn't change sign, so it is even with respect to $\pi/2$. Therefore:

\begin{lem} A Neumann eigenfunction of $\Omega_{\eps}$ associated to $\sigma_{11}(\eps)$ is given by  
$$
v(r)\doteq p_{\nu_{\eps}}(r)+p_{\nu_{\eps}}(\pi-r)
$$ 
where $\nu_{\eps}>0$ is now the lowest positive root of 
$$
p'_{\nu_\eps}(\eps)-p'_{\nu_\eps}(\pi-\eps)=0.
$$
The corresponding eigenvalue is
$$
\sigma_{11}(\eps)=\nu_{\eps}(\nu_{\eps}+1).
$$
\end{lem}
Of course the values of $\nu_{\eps}$ corresponding to $\eta_{11}(\eps)$ and $\sigma_{11}(\eps)$ are different. From now on, for simplicity of notation, we will drop the dependence of $\nu$ on $\eps$.

\smallskip

We will need  the following formula

\begin{lem}\label{apptwo} One has:
\begin{equation}\label{pprime}
p'_{\nu}(\pi-\eps)=\cos(\pi\nu)p'_{\nu}(\eps)+\frac2{\pi}\sin(\pi\nu)R_{\nu}(\cos\eps),
\end{equation}
where 
$
R_{\nu}(x)=xQ'_{\nu}(x)-\nu(\nu+1)Q_{\nu}(x).
$
\end{lem}
We included a proof of Lemma \ref{apptwo} in Appendix \ref{app2}. We have all the ingredients to prove Theorem \ref{mainasy}.
\subsection{Proof of Theorem \ref{mainasy}} We set $\lambda(\eps)=\eta_{11}(\eps)$ for $D_\eps$ and $\lambda(\eps)=\sigma_{11}(\eps)$ for $\Omega_\eps$. From Lemma \ref{lem_est_eigen} we have that, as $\eps\to 0^{+}$:
\begin{equation}\label{asylambda}
\abs{\lambda(\eps)-2}\leq C\eps^{2}
\end{equation}
where $C>0$ is a universal constant. Recall the definition $\lambda(\eps)=\nu_\eps(\nu_\eps+1)$ (note: $\nu_\eps$ depends on whether we are considering $\eta_{11}(\eps)$ or $\sigma_{11}(\eps)$). We drop from now on the dependence of $\lambda(\eps)$ and $\nu_\eps$ on $\eps$ for simplicity, and write $\lambda,\nu$. Set $\nu=1-\alpha$ so that $\alpha>0$ and 
$\lambda=(1-\alpha)(2-\alpha)$. The fact that $\alpha>0$ (for $\eps$ small) follows since $\lambda<2$, see Lemma \ref{lem_est_eigen}. It is readily seen from \eqref{asylambda} that then
$$
\alpha\leq C\eps^{2}, \quad \abs{1-\nu}\leq C\eps^{2}.
$$
Using the expansion of $P_{\nu}(x)$ in \eqref{asyP} one finds easily that then $p'_{\nu}(\eps)=1+O(\eps^{2})$ and $\cos(\pi\nu)=-1+O(\eps^{2})$.
Hence for the first term in the right-hand side of \eqref{pprime} we  have
$$
\cos(\pi\nu)p'_{\nu}(\eps)=-1+O(\eps^{2}).
$$
About the second term, observe that
$$
\dfrac 2{\pi}\sin(\pi\nu)=\dfrac 2{\pi}\sin(\pi(1-\alpha))=2\alpha+O(\alpha^{3})=2\alpha+O(\eps^{6})
$$
From the expression of $R_{\nu}(x)$ in Lemma \ref{pprime} and from the asymptotics of $Q_{\nu}(x)$ in Lemma \ref{appthree} we obtain:
$$
R_{\nu}(x)=\dfrac{x}{2(1-x)}-\dfrac{\nu(\nu+1)}{4}x\log(1-x)+\dfrac 12\nu(\nu+1)\log(1-x)+O(1).
$$
as $x\to 1^{-}$. Since $\abs{\nu-1}=\abs{\alpha}\leq C\eps^{2}$, and $\cos\eps=1-\frac{\eps^{2}}{2}+O(\eps^{4})$ as $\eps\to 0$, we get
$$
R_{\nu}(\cos\eps)=\dfrac 1{\eps^{2}}+\log\eps+O(1)
$$
as $\eps\to 0$. We conclude that
\begin{equation}\label{conclusion}
\begin{aligned}
p'_{\nu}(\pi-\eps)&=\cos(\pi\nu)p'_{\nu}(\eps)+\frac2{\pi}\sin(\pi\nu)R_{\nu}(\cos\eps)\\
&=-1+\dfrac{2\alpha}{\eps^{2}}+2\alpha\log\eps+O(\eps^{2})
\end{aligned}
\end{equation}
We can now finish the proof. We start from the expansion of $\eta_{11}(\eps)$. The Neumann condition imposes
$$
p'_{\nu}(\pi-\eps)=0.
$$
This implies immediately that 
$
\alpha=\frac{\eps^{2}}{2}+\phi(\eps)
$
where $\phi(\eps)=o(\eps^{2})$. To match the log term, one must have necessarily
$$
\phi(\eps)=-\frac 12\eps^{4}\log\eps+o(\eps^{4}|\log\eps|).
$$
Recalling that $\eta_{11}(\eps)=(1-\alpha)(2-\alpha)$ we get the first part of the theorem:
$$
\eta_{11}(\eps)=2-\frac 32\eps^{2}+\frac 32\eps^{4}\log\eps+o(\eps^{4}|\log\eps|).
$$
Let us now consider the eigenvalue $\sigma_{11}(\eps)$ of the annulus $\Omega_{\eps}$. The Neumann condition imposes 
$$
p'_{\nu}(\eps)-p'_{\nu}(\pi-\eps)=0.
$$
Recall that $p'_{\nu}(\eps)=1+O(\eps^{2})$; for $p'_{\nu}(\pi-\eps)$ we proceed as before, and there is a difference only in the constant term:
$$
p'_{\nu}(\eps)-p'_{\nu}(\pi-\eps)=
-2+\dfrac{2\alpha}{\eps^{2}}+2\alpha\log\eps+O(\eps^{2})
$$
which implies that
$$
\sigma_{11}(\eps)=2-3\eps^{2}+3\eps^{4}\log\eps+o(\eps^{4}|\log\eps|).
$$
This ends the proof.

\appendix
\section{}\label{sec_app}

\subsection{Proof of  \eqref{weightedeigenvalue} in Lemma \ref{radial_pb}}\label{app_weighted}
Let $f$ be a solution of problem \eqref{SL1}:
\begin{equation}\label{SLEQ}
\begin{cases}
-(Gf')'+\frac{4\pi^2}{G}f=\kappa f\,, & {\rm in\ }(0,M)\,,\\
\lim_{a\to 0^+}G(a)f'(a)=f'(M)=0.
\end{cases}
\end{equation}
We perform the change of variable $a=s_q(r)$, where $s_q(r)\doteq\int_{B(0,r)}\rho_q^2$. Set $v(r)=f(s_q(r))$. We have that
$$
v'(r)=f'(s_q(r))s_q'(r)=r f'(s_q(r))\int_0^{2\pi}\rho_q^2(r,\theta)d\theta;
$$
$$
v''(r)=f''(s_q(r))s'_q(r)^2+f'(s_q(r))s_q''(r).
$$
 To simplify the notation, set
$$
m(r):=\int_0^{2\pi}\rho_q^2(r,\theta)d\theta.
$$
In particular, $s_q'(r)=rm(r)$.
Then
$$
f'(s_q(r))=\frac{v'(r)}{r m(r)}
$$
and
$$
f''(s_q(r))=\frac{1}{r^2m(r)^2}(v''(r)-v'(r)\frac{(rm(r))'}{rm(r)}).
$$
The function $G(s_q(r))$ is also easily computed: 
$$
G(s_q(r))=2\pi r^2 m(r)
$$
and then
$$
G'(s_q(r))=\frac{4\pi r m(r)+2\pi r^2m'(r)}{s_q'(r)}=\frac{4\pi r m(r)+2\pi r^2m'(r)}{r m(r)}=4\pi+\frac{2\pi r m'(r)}{m(r)}.
$$
Then replacing everything in the left-hand side of \eqref{SLEQ}
\begin{multline*}
-G'(s_q(r))f'(s_q(r))-G(s_q(r))f''(s_q(r))+\frac{4\pi^2 f(s_q(r))}{G(s_q(r))}\\
=\frac{2\pi}{m(r)}(-v''(r)-\frac{v'(r)}{r}+\frac{v(r)}{r^2})
\end{multline*}
so that the equation reads
$$
-v''(r)-\frac{v'(r)}{r}+\frac{v(r)}{r^2}=\kappa\frac{m(r)}{2\pi}v(r)=\kappa\tilde\rho_q^2(r) v(r)
$$
where
$$
\tilde\rho_q^2(r)=\frac{m(r)}{2\pi}
$$
is the radialization of $\rho_q^2$, which is what we wanted. The boundary condition $f'(M)=0$ translates into $m(1)v'(1)=0$, which implies $v'(1)=0$. On the other hand, since $\lim_{r\to 0^+}s_q(r)f'(s_q(r))=0$, we get that $\lim_{r\to 0^+}rv'(r)=0$. This characterizes the eigenfunctions of the form $u=v(r)e^{\pm i\theta}$ of
$$
\Delta u=\mu \tilde\rho_q^2 u
$$
on $\mathbb D$ with Neumann boundary conditions, or, equivalently by a change of gauge, the radial eigenfunctions of $\Delta_{A_0}u=\mu\tilde\rho^2_q u$ on $\mathbb D$ with Neumann boundary conditions.

\subsection{Proof of Lemma \ref{lem_radial}}\label{app_neumanntosteklov}
We will prove \eqref{condition_convergence} for  $u\in C^{\infty}(\overline{\mathbb D})$. The result will follow by density of $C^{\infty}(\overline{\mathbb D})$ in $W^{1,p}(\mathbb D)$. As we are interested in the behavior as $\abs q\to 1$, we can assume that $|q|>1-1/e\geq \frac 12$, and let 
\begin{align}\label{R}
&R=R(|q|)\doteq 1-\frac{1}{\abs{\log(1-|q|)}},\\
&\omega_1(\abs q)\doteq \|\rho\|_{\infty}^2(1-\abs{q})^2\abs{\log(1-\abs{q})}^3,\\
&\omega_2(|q|)\doteq\frac{2}{\abs{\log(1-|q|)}}.
\end{align}
In what follows, $C_1,C_2,...$ denote positive constants not depending on $q\in\mathbb D$ (but possibly depending on $p>1$ and the volume $M$). Note that $\omega_i(\abs{q})\to 0$ as $\abs{q}\to 1$, $i=1,2$. The proof depends on the following pointwise estimate, which shows that when $q$ is close to the boundary the support of the measure 
$\tilde\rho^2_qdv_E$ concentrates in the strip $R(\abs q)<r< 1$, whose width tends to zero as $\abs q\to 1$.

\begin{lem}\label{applemma} For all $r\in [0,R(\abs q)]$ one has:
$$
\tilde\rho_q^2(r)\leq C_1\omega_1(\abs{q}).
$$
In particular, 
$$
\int_{B(0,R)}\tilde\rho_q^2 udv_E\leq C_1\omega_1(\abs q)\int_{\mathbb D}\abs u dv_E\leq C_1\omega_1(\abs q)\|u\|_{W^{1,1}(\mathbb D)}.
$$
\end{lem}

\begin{proof} We have by Lemma \ref{Phiq} $ii)$:
\begin{multline*}
\tilde\rho_q^2(r)=\frac{1}{2\pi}\int_0^{2\pi}\rho^2\left(\frac{re^{i\theta}+q}{1+\bar q r e^{i\theta}}\right)\frac{(1-|q|^2)^2}{|1+\bar q r e^{i\theta}|^4}d\theta\\
\leq\frac{\|\rho\|_{\infty}^2}{2\pi}\int_0^{2\pi}\frac{(1-|q|^2)^2}{|1+\bar q r e^{i\theta}|^4}d\theta=\|\rho\|_{\infty}^2\frac{(1-|q|^2)^2(1+r^2|q|^2)}{(1-|q|^2r^2)^3}
\end{multline*}
where the last equality is an explicit integration. The last term, for fixed $q$, is increasing in $r$, so that
$$
\tilde\rho_q^2(r)\leq\|\rho\|_{\infty}^2\frac{(1-|q|^2)^2(1+R^2|q|^2)}{(1-|q|^2R^2)^3}\leq 
2\|\rho\|_{\infty}^2\frac{(1-|q|^2)^2}{(1-|q|^2R^2)^3}\leq C_2\|\rho\|_{\infty}^2\dfrac{(1-\abs{q})^2}
{(1-\abs{q}R)^3}
$$
Now, by the definition of $R$, since $\abs q\geq 1/2$:
$$
1-\abs{q}R=1-\abs q+\frac{\abs q}{\abs{\log(1-\abs q)}}\geq \frac{1}{2\abs{\log(1-\abs q)}}
$$
so that 
$$
\|\rho\|_{\infty}^2\dfrac{(1-\abs{q})^2}{(1-\abs{q}R)^3}\leq 8\|\rho\|_{\infty}^2(1-\abs{q})^2\abs{\log(1-\abs{q})}^3=8\omega_1(\abs q)
$$
which proves the claim.
\end{proof}
Let $\mathbb D_{R,1}$ denote the annulus $\{r:R<r<1\}$. Writing $\mathbb D =B(0,R)\cup\mathbb D_{R,1}$, we have
\begin{equation}\label{part1}
\left|\int_{\mathbb D}\tilde\rho^2_q u\,dv_E-\frac{M}{2\pi}\int_{\partial \mathbb D}u\,ds\right|\\
\leq\left|\int_{B(0,R)}\tilde\rho^2_q u\, dv_E\right|+\left|\int_R^1\int_0^{2\pi}\tilde\rho^2_qr u dr d\theta-\frac{M}{2\pi}\int_0^{2\pi}u(1,\theta)d\theta\right|
\end{equation}

The first summand is bounded above as in Lemma \ref{applemma} , and in particular, by H\"older's inequality
\begin{equation}\label{est1}
\left|\int_{B(0,R)}\tilde\rho^2_q u\, dv_E\right|\leq C_1|\mathbb D|^{\frac{1}{p'}}\omega_1(|q|)\|u\|_{W^{1,p}(\mathbb D)},
\end{equation}
where $p'$ is such that $\frac{1}{p}+\frac{1}{p'}=1$.

\smallskip

Now we consider the second summand of \eqref{part1}. It is convenient to highlight the dependence of $u$ on $(r,\theta)$.
\begin{multline}\label{part2}
\left|\int_R^1\int_0^{2\pi}\tilde\rho^2_qr u dr d\theta-\frac{M}{2\pi}\int_0^{2\pi}u(1,\theta)d\theta\right|\\
=\left|\int_0^{2\pi}u(1,\theta)\left(\int_R^1\tilde\rho^2_q r dr-\frac{M}{2\pi}\right)d\theta+\int_0^{2\pi}\int_R^1\tilde\rho^2_q r\Big(u(r,\theta)-u(1,\theta)\Big)drd\theta\right|\\
\leq \int_0^{2\pi}\left| u(1,\theta)\right|\left|\int_R^1\tilde\rho^2_q r dr-\frac{M}{2\pi}\right|d\theta+\int_0^{2\pi}\int_R^1\tilde\rho^2_q r\abs{u(r,\theta)-u(1,\theta)}drd\theta.
\end{multline}
We start considering the first summand in the third line of \eqref{part2}. First note that,
since the total mass of $\tilde\rho_q^2$ is $M$, we have
$$
\int_R^1\tilde\rho^2_q rdr-\frac{M}{2\pi}=-\frac{1}{2\pi}\int_0^R\int_0^{2\pi}\tilde\rho^2_q r d\theta dr=-\frac{1}{2\pi}\int_{B(0,R)}\tilde\rho_q^2dv_E
$$
and by Lemma \ref{applemma}:
$$
\left|\int_R^1\tilde\rho^2_qrdr-\frac{M}{2\pi}\right|\leq C_2\omega_1(\abs q).
$$
Since 
$$
\int_0^{2\pi}\abs{u(1,\theta)}\,d\theta=\int_{\bd\mathbb D}\abs{u}\leq C_{Tr}\|u\|_{W^{1,1}(\mathbb D)},
$$
where $C_{Tr}$ is the trace constant of $W^{1,1}(\mathbb D)\to L^1(\partial \mathbb D)$,
we conclude that 
\begin{equation*}
\int_0^{2\pi}\left| u(1,\theta)\right|\left|\int_R^1\tilde\rho^2_q r dr-\frac{M}{2\pi}\right|d\theta\leq C_3\omega_1(\abs q)\|u\|_{W^{1,1}(\mathbb D)},
\end{equation*}
which again, by H\"older's inequality, implies
\begin{equation}\label{est2}
\int_0^{2\pi}\left| u(1,\theta)\right|\left|\int_R^1\tilde\rho^2_q r dr-\frac{M}{2\pi}\right|d\theta\leq C_3|\mathbb D|^{\frac{1}{p'}}\omega_1(\abs q)\|u\|_{W^{1,p}(\mathbb D)}.
\end{equation}

It remains to estimate the second summand in the third line of \eqref{part2}. We have that, for all $r\in (R,1)$:
\begin{equation*}
|u(r,\theta)-u(1,\theta)|\leq\int_r^1|\partial_yu(y,\theta)|dy\leq\frac{1}{ R}\int_R^1|du|rdr.
\end{equation*}
and then we see:
\begin{equation}\label{est3}
\begin{aligned}
\int_0^{2\pi}\int_R^1\tilde\rho^2_q r\abs{u(r,\theta)-u(1,\theta)}drd\theta&\leq\int_R^1\tilde\rho^2_qr\left(\int_0^{2\pi}|u(r,\theta)-u(1,\theta)|d\theta\right) dr\\
&\leq\left(\dfrac 1{2\pi}\int_0^{2\pi}\int_R^1\tilde\rho^2_q rdrd\theta\right)\left(\frac{1}{ R}\int_0^{2\pi}\int_R^1|du|rdrd\theta\right)\\
&\leq\frac{M}{2\pi R}\|u\|_{W^{1,1}(\mathbb D_{R,1})}\\
&\leq\frac{M\pi^{\frac{1}{p'}}(1-R^2)^{\frac{1}{p'}}}{2\pi R}\|u\|_{W^{1,p}(\mathbb D)}\\
&= C_4\omega_2(|q|)^{\frac{1}{p'}}\|u\|_{W^{1,p}(\mathbb D)}
\end{aligned}
\end{equation}
because the annulus $\mathbb D_{R,1}$  has Euclidean area $\pi(1-R^2)$ and in the fourth line we use H\"older inequality. Note that, by the definition of $R$, we have
$$
1-R^2\leq \dfrac{2}{\abs{\log(1-\abs q)}}=\omega_2(\abs q).
$$
Recall also that $R=R(|q|)\to 1$ as $|q|\to 1$; as $q>\frac 12$, we see that $R(\abs q)$ is uniformly bounded below. Taking into account \eqref{est1}, \eqref{est2} and \eqref{est3}, the lemma holds with 
$$
\omega_p(|q|)\doteq C_5\omega_1(|q|)+C_6\omega_2(|q|)^{\frac{p}{p-1}},
$$
which tends to zero as $q$ approaches $\bd\mathbb D$.

\subsection{Application of Brouwer fixed point Theorem}\label{fixpoint}

We recall the following well-known application of Brouwer fixed point theorem.

\begin{thm}  Assume that $V:\overline{\mathbb D}\to\mathbb C$ is a continuous vector field such that $\langle V(x),x\rangle<0$ at every point $x$ of the boundary. Then $V$ has at least a zero in $\mathbb D$.
\end{thm}
\begin{proof} Fix $\eps>0$ and consider the map $F:\overline{\mathbb D}\to\mathbb C$ given by
$$
F(x)=x+\eps V(x).
$$
We have:
$$
\abs{F(x)}^2=\abs{x}^2+2\eps\scal{x}{V(x)}+\eps^2\abs{V(x)}^2.
$$
By assumption, there is $\delta>0$ such that, on $\bd\mathbb D$
$$
\scal{V(x)}{x}\leq -\delta;
$$
let also $M=\max_{\overline{\mathbb D}}|V|$. Then, on $\bd{\mathbb D}$ on has:
$$
\abs{F(x)}^2\leq 1-2\eps\delta+\eps^2M^2.
$$
If $\eps<\frac{2\delta}{M^2}$ we see that $\abs{F(x)}^2<1$ and then $F:\overline{\mathbb D}\to\mathbb D$. By Brouwer fixed point theorem, $F$ has a fixed point $x_0\in\mathbb D$ ($F(x_0)=x_0$) and then
$V(x_0)=0$.
\end{proof}

\subsection{Proof of Lemma \ref{lem_first_annulus}}\label{app:lemma_radial_eig}
We recall the min-max principle for the eigenvalues of \eqref{ODE_annulus_k}:
\begin{equation}\label{minmax_k}
\sigma_{kj}=\min_{\substack{V\subset H^1(\eps,\pi-\eps)\\{\rm dim}\,V=j}}\max_{0\ne v\in V}\frac{\int_\eps^{\pi-\eps}\left(\sin(r)v(r)'^2+\frac{k^2}{\sin(r)}v(r)^2\right)dr}{\int_\eps^{\pi-\eps}v(r)^2\sin(r)dr}.
\end{equation}
We have that $\mu_1(\Omega_\eps)=\sigma_{01}=0$ with constant eigenfunctions. Hence, from the monotonicity with respect to $k\in\mathbb Z$ of the Rayleigh in \eqref{minmax_k}, we conclude that
$$
\mu_2(\Omega_\eps)=\min\{\sigma_{02}(\eps),\sigma_{11}(\eps)\},
$$
hence the proof is concluded provided
$$
\sigma_{11}(\eps)<\sigma_{02}(\eps)
$$
for all $\eps\in(0,\pi/2)$. For simplicity, through the rest of the proof, we drop the dependence on $\eps$ and write $\sigma_{kj}$ and $v_{kj}$ for the eigenvalues and eigenfunctions of \eqref{ODE_annulus_k}.
Let $v_{02}$ be an eigenfunction associated with $\sigma_{02}$. Set
$$
g_{02}:=v_{02}'.
$$
Then one checks that
$$
-g_{02}''(r)-\cot(r)g_{02}'(r)+\frac{1}{\sin^2(r)}g_{02}(r)=\sigma_{02} g_{02}(r),
$$
and this follows by differentiating the differential equation in \eqref{ODE_annulus_k} with $k=0$, $v=v_{02}$, $\sigma=\sigma_{02}$. At the endpoints we have
$$
g_{02}(\eps)=g_{02}(\pi-\eps)=0.
$$
This means that $\sigma_{02}$ is an eigenvalue of the following Dirichlet problem:
\begin{equation}\label{ODE_annulus_1_2}
\begin{cases}
-v''(r)-\cot(r)v'(r)+\frac{1}{\sin^2(r)}v(r)=\lambda v(r)\,, & {\rm in\ }(\eps,\pi-\eps)\,,\\
v(\eps)=v(\pi-\eps)=0.
\end{cases}
\end{equation}
In particular
$$
\sigma_{02}\geq\lambda_{11}
$$
where $\lambda_{11}$ is the first eigenvalue of \eqref{ODE_annulus_1_2}. Recall that $\sigma_{11}$ is the first eigenvalue of the Neumann problem:
\begin{equation}
\begin{cases}
-v''(r)-\cot(r)v'(r)+\frac{1}{\sin^2(r)}v(r)=\mu v(r)\,, & {\rm in\ }(\eps,\pi-\eps)\,,\\
v'(\eps)=v'(\pi-\eps)=0.
\end{cases}
\end{equation}
The proof is concluded observing that:
$$
\sigma_{02}\geq\lambda_{11}>\sigma_{11}.
$$
\qed

\subsection{Proof of Lemma \ref{appthree}}\label{app_Q}

 Given $P_{\nu}(x)$, the Frobenius method gives a second linearly independent solution, (see \cite[2.7.6]{nist}) of the form
$$
w(x)+P_{\nu}(x)\log (1-x),
$$
where $w(x)$ is analytic in $x=1$ and $w(x)=O(1-x)$ as $x\to 1^{-}$. Therefore, the general solution is 
$$
A\Big(w(x)+P_{\nu}(x)\log (1-x)\Big)+BP_{\nu}(x),
$$
with $A,B\in\mathbb R$. The definition \eqref{explicitQ} of $Q_\nu$ fixes the constants $A,B$. More precisely, $Q_\nu$ is the unique solution which behaves as follows (see \cite[14.8.3]{nist}):
$$
Q_\nu(x)=-\frac{1}{2}\log(1-x)+c(\nu)+O((1-x)|\log(1-x)|)
$$
as $x\to 1^-$. Here $c(\nu)=\frac{1}{2}\log(2)-\gamma-\psi(\nu+1)$, where $\gamma$ is Euler's constant and $\psi(z)\doteq\frac{\Gamma'(z)}{\Gamma(z)}$. Hence $A=-\frac 12$ and $B=c(\nu)$. Now the conclusion follows straightforwardly by using the expansion of $P_\nu(x)$ in \eqref{asyP}. \qed

\subsection{Proof of Lemma \ref{apptwo}}\label{app2}

Recall that, for $r\in (0,\pi)$ we have, by definition:
$$
p_{\nu}(r)=\sin r P'_{\nu}(\cos r).
$$
Differentiating with respect to $r$, setting $x=\cos r$ and using the identity defining $P_{\nu}$:
$$
(1-x^{2})P''_{\nu}(x)-2xP'_{\nu}(x)+\nu(\nu+1)P_{\nu}(x)=0
$$
we get 
\begin{equation}\label{appzero}
p'_{\nu}(r)=-xP'_{\nu}(x)+\nu(\nu+1)P_{\nu}(x).
\end{equation}
Since $\cos(\pi-r)=-x$ we get:

\begin{equation}\label{appfirst}
p'_{\nu}(\pi-r)=xP'_{\nu}(-x)+\nu(\nu+1)P_{\nu}(-x).
\end{equation}
From \cite[14.9.10]{nist} we have
\begin{equation}\label{appsecond}
P_{\nu}(-x)=\cos(\pi\nu)P_{\nu}(x)-\frac 2{\pi}\sin(\pi\nu)Q_{\nu}(x)
\end{equation}
 hence
 \begin{equation}\label{appthird}
P'_{\nu}(-x)=-\cos(\pi\nu)P'_{\nu}(x)+\frac 2{\pi}\sin(\pi\nu)Q'_{\nu}(x)
\end{equation}
Substituting \eqref{appsecond} and  \eqref{appthird} in \eqref {appfirst} and taking into account \eqref{appzero} we get the final identity.\qed

\begin{figure}
     \centering

     \begin{subfigure}[]{0.48\textwidth}
         \centering
         \includegraphics[width=\textwidth]{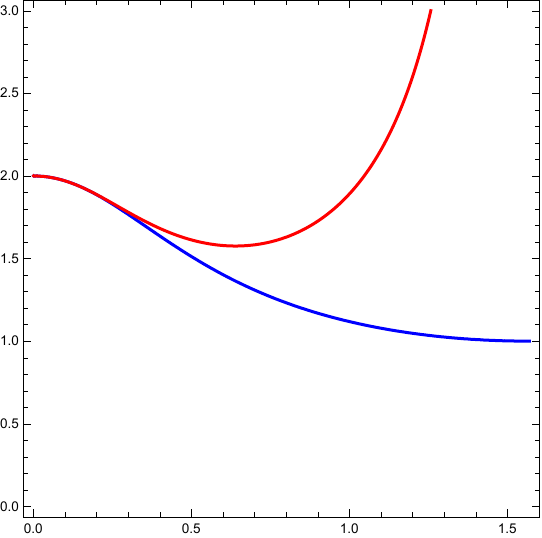}
         \caption{{\color{red}$\mu_2(\Omega^\star_\eps)$} and {\color{blue}$\mu_2(\Omega_\eps)$}}
         \label{A}
     \end{subfigure}
     \hfill
     \begin{subfigure}[]{0.48\textwidth}
         \centering
         \includegraphics[width=\textwidth]{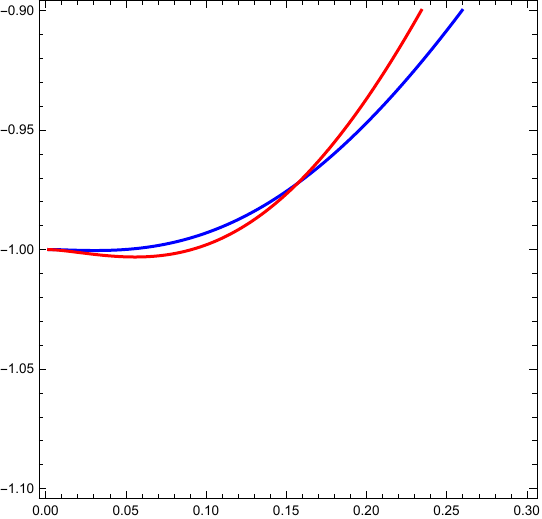}
         \caption{{\color{red}$\frac{\mu_2(\Omega^\star_\eps)-2}{3\eps^2}$} and {\color{blue}$\frac{\mu_2(\Omega_\eps)-2}{3\eps^2}$}}
         \label{B}
        
     \end{subfigure}
        \caption{The intersection of the two branches in Figure \ref{mathematica} (B) is numerically computed to happen at $\eps\approx 0.1571$, which corresponds to an area of about 98.77\% that of $\mathbb S^2$.}
        \label{mathematica}
\end{figure}

\bibliography{bibliography.bib}

@book{nist,
 editor = {Olver, Frank W. J. and Lozier, Daniel W. and Boisvert, Ronald F. and Clark, Charles W.},
 title = {{NIST} handbook of mathematical functions},
 isbn = {978-0-521-19225-5; 978-0-521-14063-8},
 year = {2010},
 publisher = {Cambridge: Cambridge University Press},
 language = {English},
 zbMATH = {5765058},
 Zbl = {1198.00002}
}

@article{MPS25,
    AUTHOR = {Michetti, Marco and Provenzano, Luigi and Savo, Alessandro},
      TITLE = {Isoperimetric inequalities and sharp upper bounds for {A}haronov-{B}ohm eigenvalues on surfaces},
   JOURNAL = {	arXiv:2604.11718},
  FJOURNAL = {	arXiv:2604.11718},
      YEAR = {2026},
}

@article {lang_laug_robin,
    AUTHOR = {Langford, Jeffrey J. and Laugesen, Richard S.},
     TITLE = {Maximizing the second {R}obin eigenvalue of simply connected
              curved membranes},
   JOURNAL = {Comput. Methods Funct. Theory},
  FJOURNAL = {Computational Methods and Function Theory},
    VOLUME = {25},
      YEAR = {2025},
    NUMBER = {1},
     PAGES = {83--117},
      ISSN = {1617-9447,2195-3724},
   MRCLASS = {35P15 (35J25 35R01 58J50)},
  MRNUMBER = {4887603},
MRREVIEWER = {Srinivasan\ Kesavan},
       DOI = {10.1007/s40315-023-00516-1},
       URL = {https://doi.org/10.1007/s40315-023-00516-1},
}

@article {GKL,
    AUTHOR = {Girouard, Alexandre and Karpukhin, Mikhail and Lagac\'e, Jean},
     TITLE = {Continuity of eigenvalues and shape optimisation for {L}aplace
              and {S}teklov problems},
   JOURNAL = {Geom. Funct. Anal.},
  FJOURNAL = {Geometric and Functional Analysis},
    VOLUME = {31},
      YEAR = {2021},
    NUMBER = {3},
     PAGES = {513--561},
      ISSN = {1016-443X,1420-8970},
   MRCLASS = {49R05 (35P05 46E35 49Q10)},
  MRNUMBER = {4311579},
MRREVIEWER = {Andrzej\ M.\ My\'sli\'nski},
       DOI = {10.1007/s00039-021-00573-5},
       URL = {https://doi.org/10.1007/s00039-021-00573-5},
}

@article {CLPS25,
    AUTHOR = {Colbois, Bruno and L\'ena, Corentin and Provenzano, Luigi and
              Savo, Alessandro},
     TITLE = {A reverse {F}aber-{K}rahn inequality for the magnetic
              {L}aplacian},
   JOURNAL = {J. Math. Pures Appl. (9)},
  FJOURNAL = {Journal de Math\'ematiques Pures et Appliqu\'ees. Neuvi\`eme
              S\'erie},
    VOLUME = {192},
      YEAR = {2024},
     PAGES = {Paper No. 103632, 33},
      ISSN = {0021-7824,1776-3371},
   MRCLASS = {35P15 (35J25 49Q10 49R05 81Q10)},
  MRNUMBER = {4824640},
       DOI = {10.1016/j.matpur.2024.103632},
       URL = {https://doi.org/10.1016/j.matpur.2024.103632},
}

@article {bucur_gafa,
    AUTHOR = {Bucur, Dorin and Laugesen, Richard S. and Martinet, Eloi and
              Nahon, Micka\"el},
     TITLE = {Spherical caps do not always maximize {N}eumann eigenvalues on
              the sphere},
   JOURNAL = {Geom. Funct. Anal.},
  FJOURNAL = {Geometric and Functional Analysis},
    VOLUME = {35},
      YEAR = {2025},
    NUMBER = {5},
     PAGES = {1313--1345},
      ISSN = {1016-443X,1420-8970},
   MRCLASS = {35P15 (58J50)},
  MRNUMBER = {4990469},
       DOI = {10.1007/s00039-025-00721-1},
       URL = {https://doi.org/10.1007/s00039-025-00721-1},
}

@article {laugesen,
    AUTHOR = {Langford, Jeffrey J. and Laugesen, Richard S.},
     TITLE = {Maximizers beyond the hemisphere for the second {N}eumann
              eigenvalue},
   JOURNAL = {Math. Ann.},
  FJOURNAL = {Mathematische Annalen},
    VOLUME = {386},
      YEAR = {2023},
    NUMBER = {3-4},
     PAGES = {2255--2281},
      ISSN = {0025-5831,1432-1807},
   MRCLASS = {35P15 (58J50)},
  MRNUMBER = {4612417},
MRREVIEWER = {Panagiotis\ Polymerakis},
       DOI = {10.1007/s00208-022-02455-z},
       URL = {https://doi.org/10.1007/s00208-022-02455-z},
}

@article {bandle2,
    AUTHOR = {Bandle, Catherine},
     TITLE = {Isoperimetric inequality for some eigenvalues of an
              inhomogeneous, free membrane},
   JOURNAL = {SIAM J. Appl. Math.},
  FJOURNAL = {SIAM Journal on Applied Mathematics},
    VOLUME = {22},
      YEAR = {1972},
     PAGES = {142--147},
      ISSN = {0036-1399},
   MRCLASS = {35P15 (52A40 73.35)},
  MRNUMBER = {313648},
MRREVIEWER = {D.\ O.\ Banks},
       DOI = {10.1137/0122016},
       URL = {https://doi.org/10.1137/0122016},
}

@article {CPS22,
    AUTHOR = {Colbois, Bruno and Provenzano, Luigi and Savo, Alessandro},
     TITLE = {Isoperimetric inequalities for the magnetic {N}eumann and
              {S}teklov problems with {A}haronov-{B}ohm magnetic potential},
   JOURNAL = {J. Geom. Anal.},
  FJOURNAL = {Journal of Geometric Analysis},
    VOLUME = {32},
      YEAR = {2022},
    NUMBER = {11},
     PAGES = {Paper No. 285, 38},
      ISSN = {1050-6926,1559-002X},
   MRCLASS = {35J10 (35P15 49R05 58J50 81Q10)},
  MRNUMBER = {4478480},
       DOI = {10.1007/s12220-022-01001-2},
       URL = {https://doi.org/10.1007/s12220-022-01001-2},
}

@article {martinet_sphere,
    AUTHOR = {Martinet, Eloi},
     TITLE = {Numerical optimization of {N}eumann eigenvalues of domains in
              the sphere},
   JOURNAL = {J. Comput. Phys.},
  FJOURNAL = {Journal of Computational Physics},
    VOLUME = {508},
      YEAR = {2024},
     PAGES = {Paper No. 113002, 26},
      ISSN = {0021-9991,1090-2716},
   MRCLASS = {65N25},
  MRNUMBER = {4733245},
       DOI = {10.1016/j.jcp.2024.113002},
       URL = {https://doi.org/10.1016/j.jcp.2024.113002},
}

@article {bucur_sphere_2,
    AUTHOR = {Bucur, Dorin and Martinet, Eloi and Nahon, Micka\"el},
     TITLE = {Sharp inequalities for {N}eumann eigenvalues on the sphere},
   JOURNAL = {J. Differential Geom.},
  FJOURNAL = {Journal of Differential Geometry},
    VOLUME = {131},
      YEAR = {2025},
    NUMBER = {1},
     PAGES = {43--63},
      ISSN = {0022-040X,1945-743X},
   MRCLASS = {35P15},
  MRNUMBER = {4947549},
MRREVIEWER = {Richard\ S.\ Laugesen},
       DOI = {10.4310/jdg/1755542131},
       URL = {https://doi.org/10.4310/jdg/1755542131},
}

@book {FH_book,
    AUTHOR = {Fournais, S\oren and Helffer, Bernard},
     TITLE = {Spectral methods in surface superconductivity},
    SERIES = {Progress in Nonlinear Differential Equations and their
              Applications},
    VOLUME = {77},
 PUBLISHER = {Birkh\"{a}user Boston, Inc., Boston, MA},
      YEAR = {2010},
     PAGES = {xx+324},
      ISBN = {978-0-8176-4796-4},
   MRCLASS = {35-02 (35P15 35Q56 47F05 47N50 49N60 82D55)},
  MRNUMBER = {2662319},
MRREVIEWER = {Yuri A. Kordyukov},
}

@article {szego_1,
    AUTHOR = {Szeg\"{o}, G.},
     TITLE = {Inequalities for certain eigenvalues of a membrane of given
              area},
   JOURNAL = {J. Rational Mech. Anal.},
  FJOURNAL = {Journal of Rational Mechanics and Analysis},
    VOLUME = {3},
      YEAR = {1954},
     PAGES = {343--356},
      ISSN = {1943-5282},
   MRCLASS = {36.0X},
  MRNUMBER = {61749},
MRREVIEWER = {E. T. Copson},
       DOI = {10.1512/iumj.1954.3.53017},
       URL = {https://doi.org/10.1512/iumj.1954.3.53017},
}

@article {dallaprovenzano,
    AUTHOR = {Dalla Riva, Matteo and Provenzano, Luigi},
     TITLE = {On vibrating thin membranes with mass concentrated near the
              boundary: an asymptotic analysis},
   JOURNAL = {SIAM J. Math. Anal.},
  FJOURNAL = {SIAM Journal on Mathematical Analysis},
    VOLUME = {50},
      YEAR = {2018},
    NUMBER = {3},
     PAGES = {2928--2967},
      ISSN = {0036-1410,1095-7154},
   MRCLASS = {35J25 (35B25 35C20 35P05 70Jxx 74K15)},
  MRNUMBER = {3813233},
MRREVIEWER = {Mario\ M.\ Coclite},
       DOI = {10.1137/17M1118221},
       URL = {https://doi.org/10.1137/17M1118221},
}

@preamble{
   "\def\cprime{$'$} "
}

@article {lambertiprovenzano2,
    AUTHOR = {Lamberti, Pier Domenico and Provenzano, Luigi},
     TITLE = {Neumann to {S}teklov eigenvalues: asymptotic and monotonicity
              results},
   JOURNAL = {Proc. Roy. Soc. Edinburgh Sect. A},
  FJOURNAL = {Proceedings of the Royal Society of Edinburgh. Section A.
              Mathematics},
    VOLUME = {147},
      YEAR = {2017},
    NUMBER = {2},
     PAGES = {429--447},
      ISSN = {0308-2105,1473-7124},
   MRCLASS = {35J05 (33C10 35B25 35C20 35J25 35P15)},
  MRNUMBER = {3627957},
       DOI = {10.1017/S0308210516000214},
       URL = {https://doi.org/10.1017/S0308210516000214},
}

@book {chavel,
    AUTHOR = {Chavel, Isaac},
     TITLE = {Eigenvalues in {R}iemannian geometry},
    SERIES = {Pure and Applied Mathematics},
    VOLUME = {115},
      NOTE = {Including a chapter by Burton Randol,
              With an appendix by Jozef Dodziuk},
 PUBLISHER = {Academic Press, Inc., Orlando, FL},
      YEAR = {1984},
     PAGES = {xiv+362},
      ISBN = {0-12-170640-0},
   MRCLASS = {58G25 (35P99 53C20)},
  MRNUMBER = {768584},
MRREVIEWER = {G{\'e}rard Besson},
}

@book {kato1,
    AUTHOR = {Kato, Tosio},
     TITLE = {Perturbation theory for linear operators},
   EDITION = {Second},
      NOTE = {Grundlehren der Mathematischen Wissenschaften, Band 132},
 PUBLISHER = {Springer-Verlag, Berlin-New York},
      YEAR = {1976},
     PAGES = {xxi+619},
   MRCLASS = {47-XX},
  MRNUMBER = {0407617 (53 \#11389)},
}

@incollection{lambertiprovenzano1,
year={2015},
isbn={978-3-319-12576-3},
booktitle={Current Trends in Analysis and Its Applications},
series={Trends in Mathematics},
editor={Mityushev, Vladimir V. and Ruzhansky, Michael V.},
doi={10.1007/978-3-319-12577-0_21},
title={Viewing the {S}teklov Eigenvalues of the {L}aplace Operator as Critical {N}eumann Eigenvalues},
url={http://dx.doi.org/10.1007/978-3-319-12577-0_21},
publisher={Springer International Publishing},
author={Lamberti, PierDomenico and Provenzano, Luigi},
pages={171-178},
language={English}
}

@article {weinberger,
    AUTHOR = {Weinberger, H. F.},
     TITLE = {An isoperimetric inequality for the {$N$}-dimensional free
              membrane problem},
   JOURNAL = {J. Rational Mech. Anal.},
    VOLUME = {5},
      YEAR = {1956},
     PAGES = {633--636},
   MRCLASS = {52.0X},
  MRNUMBER = {0079286 (18,63c)},
MRREVIEWER = {P. Funk},
}

@article {ash_beng,
    AUTHOR = {Ashbaugh, Mark S. and Benguria, Rafael D.},
     TITLE = {Sharp upper bound to the first nonzero {N}eumann eigenvalue
              for bounded domains in spaces of constant curvature},
   JOURNAL = {J. London Math. Soc. (2)},
  FJOURNAL = {Journal of the London Mathematical Society. Second Series},
    VOLUME = {52},
      YEAR = {1995},
    NUMBER = {2},
     PAGES = {402--416},
      ISSN = {0024-6107},
   MRCLASS = {35P15 (35J05 58G25)},
  MRNUMBER = {1356151},
MRREVIEWER = {Karl-Josef Witsch},
       DOI = {10.1112/jlms/52.2.402},
       URL = {https://doi.org/10.1112/jlms/52.2.402},
}

@article {chavel_isop,
    AUTHOR = {Chavel, Isaac and Feldman, Edgar A.},
     TITLE = {Isoperimetric inequalities on curved surfaces},
   JOURNAL = {Adv. in Math.},
  FJOURNAL = {Advances in Mathematics},
    VOLUME = {37},
      YEAR = {1980},
    NUMBER = {2},
     PAGES = {83--98},
      ISSN = {0001-8708},
   MRCLASS = {53C65 (49G05)},
  MRNUMBER = {591721},
MRREVIEWER = {R.\ Osserman},
       DOI = {10.1016/0001-8708(80)90028-6},
       URL = {https://doi.org/10.1016/0001-8708(80)90028-6},
}

@article {Da,
    AUTHOR = {Daners, Daniel},
     TITLE = {A {F}aber-{K}rahn inequality for {R}obin problems in any space
              dimension},
   JOURNAL = {Math. Ann.},
  FJOURNAL = {Mathematische Annalen},
    VOLUME = {335},
      YEAR = {2006},
    NUMBER = {4},
     PAGES = {767--785},
      ISSN = {0025-5831},
   MRCLASS = {35P15 (35J25)},
  MRNUMBER = {2232016},
MRREVIEWER = {Siegfried Carl},
       DOI = {10.1007/s00208-006-0753-8},
       URL = {https://doi.org/10.1007/s00208-006-0753-8},
}

@article {lan_lau,
    AUTHOR = {Langford, Jeffrey J. and Laugesen, Richard S.},
     TITLE = {Maximizers beyond the hemisphere for the second {N}eumann
              eigenvalue},
   JOURNAL = {Math. Ann.},
  FJOURNAL = {Mathematische Annalen},
    VOLUME = {386},
      YEAR = {2023},
    NUMBER = {3-4},
     PAGES = {2255--2281},
      ISSN = {0025-5831},
   MRCLASS = {35P15 (58J50)},
  MRNUMBER = {4612417},
MRREVIEWER = {Panagiotis Polymerakis},
       DOI = {10.1007/s00208-022-02455-z},
       URL = {https://doi.org/10.1007/s00208-022-02455-z},
}
\bibliographystyle{abbrv}

\end{document}